\documentclass[12pt,index]{amsart}
\usepackage[pagebackref,colorlinks,linkcolor=red,citecolor=blue,urlcolor=blue,hypertexnames=true]{hyperref}
\usepackage{color}
\RequirePackage[l2tabu, orthodox]{nag}
\usepackage[margin=1in]{geometry}
\usepackage{amsmath, amsthm, amssymb}
\usepackage{enumerate}
\usepackage{mathrsfs} 
\usepackage[normalem]{ulem}
\newtheorem{thm}{Theorem}
\newtheorem{lemma}{Lemma}
\newtheorem{prop}{Proposition}
\newtheorem{conjecture}{Conjecture}
\theoremstyle{definition}
\newtheorem{example}{Example}
\newtheorem{definition}{Definition}
\newtheorem{remark}{Remark}

\newcommand{\mat}[4]{\begin{pmatrix}{#1}&{#2}\\{#3}&{#4} \end{pmatrix}}

\newcommand{\J}{\mathcal{J}}
\newcommand{\Btwo}[1]{B_2({#1}^{*},{ #1})}
\newcommand{\Bone}[1]{B_1({#1}^{*},{ #1})}

\newcommand{\T}{\mathcal{T}}

\newcommand{\Y}{\mathcal{Y}}

\newcommand{\tw}[1]{\widetilde{#1}}
\renewcommand{\L}{\mathcal{L}}
\newcommand{\mc}[1]{\mathcal{#1}}
\renewcommand{\a}{\alpha}
\renewcommand{\b}{\beta}
\newcommand{\ip}[2]{\left \langle #1,\ #2\right \rangle}
\newcommand{\ipsq}[2]{\left [ #1, #2\right ]}
\newcommand{\hef}{h\otimes f_j\otimes f_k}
\newcommand{\heg}{h\otimes f_a\otimes f_b}
\newcommand{\gab}{g\otimes f_a\otimes f_b}
\newcommand{\hab}{h\otimes f_a\otimes f_b}
\newcommand{\hefp}{h^\prime \otimes f_{j^\prime}\otimes f_{k^\prime}}
\newcommand{\sym}{\operatorname{Sym}}
\newcommand{\ran}{\operatorname{ran}}

\newcommand{\sB}{\mathscr{B}}
\newcommand{\tay}{\operatorname{Tay}}
\newcommand{\df}[1]{{\textit{#1}}{\index{#1}}}
\newcommand{\ext}{\operatorname{Ext}}
\newcommand{\co}{\operatorname{co}}
\makeindex

\title[Lifting $3$-isometric Tuples]{Lifting Commuting 3-Isometric Tuples}

\author[Russo]{Benjamin Russo}
 \address{Benjamin Russo, Department of Mathematics\\
  University of Florida, Gainesville 
   }
   \email{russo5@.ufl.edu}


\date{\today}

\begin{document}
\maketitle
\begin{abstract}
An operator $T$ is called a 3-isometry if there exists operators $B_1(T^*,T)$ and $B_2(T^*,T)$ such that 
\[Q(n)=T^{*n}T^n=1+nB_1(T^*,T)+n^2 B_2(T^*,T)\]
for all natural numbers $n$. An operator $J$ is a Jordan operator of order $2$ if $J=U+N$ where $U$ is unitary, $N$ is nilpotent order $2$, and $U$ and $N$ commute.  An easy computation shows that $J$ is a $3$-isometry and that the restriction of $J$ to an invariant subspace is also a $3$-isometry. Those $3$-isometries which are the restriction of a Jordan operator to an invariant subspace can be identified, using the theory of completely positive maps, in terms of a positivity condition on the operator pencil $Q(s).$ In this article,  we establish the analogous result in  the multi-variable setting and show, by modifying an example of Choi, that an additional hypothesis is necessary.  Lastly we discuss the joint spectrum of sub-Jordan tuples and derive results for 3-symmetric operators as a corollary. 
\end{abstract}

\section{Introduction}
Let $H$ denote a complex Hilbert space and $\sB(H)$ the bounded linear operators on $H$.
An operator  $T$ on $H$ is a \df{3-isometry} if 
\[T^{*3}T^3-3T^{*2}T^2+3T^*T-I=0.\]
Equivalently an operator $T$ is a 3-isometry  if there exist operators $\Bone{T}, \Btwo{T} \in \sB(H)$ such that, 
\begin{equation}\label{eq:3iso}
T^{*n}T^n=I+n\Bone{T}+n^2\Btwo{T}
\end{equation}
 for positive integers $n$. 
Similarly, $\T\in \sB(H)$ is a  \df{3-symmetric operator} if
\begin{equation}\label{eq:3sym}
\exp(-is\T^*)\exp(is\T)=I+s\Bone{\T}+s^2\Btwo{\T} 
\end{equation}
for some $\Bone{\T}$, $\Btwo{\T}\in \sB(H)$ and all real numbers $s$. 
 In particular, if  $\T$ is a 3-symmetric operator, then  $T=\exp(i\T)$ is a 3-isometric operator.\\

An operator $J$ is \df{$s$-Jordan (of order $2$)} if $J=S+N$, where $S$ and $N$ commute, $N$ is nilpotent order two, and $S$ is self-adjoint. A calcuation shows  $J$ is an example of a 3-symmetric operator.
 Similarly  $J$ is \df{$u$-Jordan (of order $2$)} if $J=U+N$, where $U$ and $N$ commmute, $U$ is unitary, and $N$ is nilpotent of order two.  One can check that $u$-Jordan operators are 3-isometric and if $J$ is an $s$-Jordan operator, then $\exp(iJ)$ is $u$-Jordan. For the remainder of the paper we will refer to $u$-Jordan and $s$-Jordan operators as simply Jordan when it is clear from context which type is being discussed.
 
  An operator $T$ on a Hilbert space $H$ has an \df{extension} or \df{lifts} to an operator $J$ on a Hilbert space 
  if there is an isometry $V:H\to K$ such that $VT=JV$.  If $J$ is 3-isometric (resp. 3-symmetric) and $T$ lifts to 
  $J$, then $T$ is 3-isometric (resp. 3-symmetric) since, in that case,
\[
  T^{*n}T^n = V^* J^{*n}J^nV
\]
 and the right hand side is quadratic in $n$.

\begin{thm}\label{3symlift}
$\T \in \sB(H)$ is a 3-symmetric operator if and only if $\T$ has an extension to an operator of the form 
\[\J=\mat{A}{\lambda 1}{0}{A}\]
where $A$ is self-adjoint and $\lambda \in \mathbb{C}.$
\end{thm}

Agler established Theorem \ref{3symlift} in the general case in \cite{Aglerthesis}. A preliminary version of the result was initially proven by Helton in \cite{Helton_2}.

The notation $A\succeq 0$ indicates that the operator $A$ on Hilbert space is positive semidefinite.
Given $c>0$,  let $\mathfrak{F}_c$ denote the class of  3-isometric operators $T$ such that
\[\hat{Q}(T,s):=I+s\Bone{T}+s^2\Btwo{T}-\frac{1}{c^2}\Btwo{T}\succeq0 \]
for all $s\in  \mathbb{R}$.
\begin{thm}\cite{McCullough}[3-isometric lifting theorem]\label{3isothm}
An operator $T$ on a Hilbert space $H$ is in the class $\mathfrak{F}_c$ if and only if there is
 a unitary operator $U$ on a Hilbert space $K$  and an isometry $V:H\to K\oplus K$ such that
  $VT=JV$, where 
\[J=\mat{U}{cU}{0}{U}.\] 
Moreover, if $T$ is invertible, then,  $VT^{-1}=J^{-1}V,$ the spectrum of $T$ is a subset of the unit circle, and $U$ can be chosen so that $\sigma(T)=\sigma(U)=\sigma(J)$.
\end{thm}

By use of a functional calculus argument Theorem \ref{3symlift} can be recovered from Theorem \ref{3isothm}.\\
 
In the case of tuples of 3-symmetric and 3-isometric operators, the picture is not as clear. Ball and Helton \cite{BallandHelton} first considered a natural simplification of the problem. Let \[\{J_n=S_n+N_n\}\] be a finite collection of commuting Jordan operators such that the nilpotent parts have the following relation, 
\[N_iN_j=0\] 
for all $i$ and $j$ and the $S_n$ are self-adjoint. We will call this a \df{commuting Jordan family}. Let $\{T_n\}$ be a finite collection of commuting 3-symmetric operators that satisfy the following,
\[Q(s)=e^{-is_kT_k^*}\ldots e^{-is_1T_1^*}e^{is_1T_1}\ldots e^{is_k T_k}=\sum_{\substack{j_1,\ldots j_k\\j_1+\dots+j_k\leq 2}}
B_{j_1,\ldots ,j_k}s_1^{j_1}\ldots s_k^{j_k}.\]
We will call this a {\it commuting family of 3-symmetric operators}. 
\begin{conjecture}\cite{BallandHelton}\label{3symconjecture}
A collection of operators $\{T_n\}$ can be extended to a commuting Jordan family $\{J_n\}$ if and only if $\{T_n\}$ is a commuting family of 3-symmetric operators. 
\end{conjecture}
Ball and Helton established this result using disconjugacy theory for multivariable Sturm-Liouville operators for tuples  $T$ of 3-isometric operators with a cyclic vector and satisfying a certain smoothness hypothesis. 
In this paper we show that an analog of this conjecture for tuples of 3-isometric operators is false and give a counter-example.
\begin{definition}
A commuting 2-tuple of operators $T=(T_1,T_2)$ is a 2-tuple of 3-isometries if there exists bounded operators $B_{i,j}$ for $0\le i+j\le 2$ (and $i,j\ge 0$) such that 
 \[Q_{T}(n,m)=T_2^{*m}T_1^{*n}T_1^nT_2^m=\sum_{0\le  i+j\le 2}  m^in^jB_{i,j}\]
for all $(n,m)\in\mathbb{N}$. We will call $Q_T$ the associated \df{quadratic pencil}.
\end{definition}

\begin{definition}
Fix positive real numbers $c,d$. A 2-tuple of commuting 3-isometries $T=(T_1,T_2)$ is in the class $\mathfrak{F}_{(c,d)}$ if 
\[\hat{Q}_{T}(\alpha,\beta)=Q_{T}(\alpha,\beta)-\frac{1}{c^2}{B}_{2,0}-\frac{1}{d^2}{B}_{0,2} \succeq 0\]
for all $(\alpha, \beta)\in \mathbb{R}^2$.
\end{definition}

 The following definition identifies a canonical class of model operators for the class $\mathcal{F}_{c,d}$. 
\begin{definition}
Given $c,d>0$ a 2-tuple $J=(J_1, J_2)$ is in the class $\mathfrak{J}_{c,d}$ if 
\begin{equation}\label{eq:Jforms}
J_1=\begin{pmatrix}U_1 & c U_1& 0\\ 0 & U_1 & 0\\ 0 & 0 & U_1\end{pmatrix},
\ \ \ \ 
J_2=\begin{pmatrix}U_2 & 0&d U_2\\0 & U_2 & 0\\ 0 & 0 & U_2\end{pmatrix}.
\end{equation}
for some unitary operators $U_1$, $U_2$ that commute. 
\end{definition}

Given $J\in\mathfrak{J}_{c,d}$, compute, for non-negative integers $m,n$, 
\[
J_1^n=\begin{pmatrix}U_1^n & nc U_1^n& 0\\ 0 & U_1^n & 0\\ 0 & 0 & U_1^n\end{pmatrix}, \ \ \ 
J_2^m=\begin{pmatrix}U_2^m & 0&md U_2^m\\0 & U_2^m & 0\\ 0 & 0 & U_2^m\end{pmatrix}
\]
and 
\begin{equation}\label{eq:J_1J_2}
J_2^{*m}J_1^{*n}J_1^nJ_2^m=\begin{pmatrix}1& nc & md\\ nc & n^2c^2+1 & ncmd\\ md & ncmd & m^2d^2+1\end{pmatrix}. 
\end{equation}
It follows that $\mathfrak{J}_{c,d}\subseteq \mathfrak{F}_{c,d}.$

\begin{thm}\label{liftthm} A 3-isometric 2-tuple  $T=(T_1, T_2)$  in the  class $\mathfrak{F}_{c,d}$  lifts to a 2-tuple $J=(J_1, J_2)$ in the class $\mathfrak{J}_{c,d}$ if and only if the the quadratic pencil $\hat{Q}_T(\alpha,\beta)$ 
factors in the form, 
\[\hat{Q}_{T}(\alpha,\beta)=(V_0+\alpha V_1 + \beta V_2)^*(V_0+\alpha V_1 + \beta V_2)\]
for some operators $V_0$, $V_1$ and $V_2$ in $\sB(H)$. 
\end{thm}
\noindent Theorem \ref{liftthm} is proved in Section \ref{sec:ext}.

The proof of the first part of the following remark for $3$-symmetric operators appears  in \cite{BallandHelton}. The proof of the result for $3$-isometries is similar. The proof of the second part of the remark can be found in Section \ref{sec:counter}. 
\begin{remark}\label{remark}
 If $H$ is finite dimensional and $T\in \mathfrak{F}_{c,d}$, then $T$ is a pair of commuting $u$-Jordan operators
 and the sufficient condition of Theorem \ref{liftthm} is easily verified.  Otherwise $H$ is infinite dimensional
 and  $\hat{Q}_T$ factors in the form above with $V_j:H\to \mathcal H$, where $\mathcal H$ is an auxiliary Hilbert space,
 if and only if it factors with $V_j\in \sB(H)$.
\end{remark}

Section \ref{sec:counter} exhibits, by construction, a 3-isometric 2-tuple $T$ in the class $\mathfrak{F}_{c,d}$ 
 for which $\hat{Q}_T$ does not factor (in the form given in Theorem \ref{liftthm}). We show that this $T$ does not lift to a $J\in \mathfrak{J}_{cd}$ and further that $T$ does not lift to any Jordan operator in any class $\mathfrak{J}_{\tilde{c},\tilde{d}}$ for any $\tilde{c}$ and $\tilde{d}$.
In this sense the 3-isometric analog of the conjecture of Ball-Helton is false. In Section \ref{sec:spec}  we show, by a functional calculus argument, that a 2-tuple of 3-symmetric operators lift if and only if its  associated operator polynomial factors.

\section{Extensions of Theorems} 
 \label{sec:ext}
We begin by extending the results found in \cite{McCullough} to 2-tuples of invertible commuting 3-isometries 
in  $\mathfrak{F}_{c,d}$. While the proofs only deal with 2-tuples, the extension to general $n$-tuples is apparent.

A subspace $A$ of $\sB(H)$ is \df{unital} if it contains the identity and is  \df{self-adjoint} if $T\in A$ implies $T^*\in A$. For a given $N\in \mathbb{N}$, let $M_N(\mathbb{C})$ be the space of $N\times N$ matrices with complex entries, denoted $M_N$ when the context is clear. Moreover, we denote with $M_N(A)$ the space of $N\times N$ \index{$M_N(A)$} matrices with entries from $A$. Note $M_N(A)$ can be identified with a subspace of the bounded operators on $H^{(N)}=H\oplus \cdots \oplus H$ ($N$-copies)
 as well as with $M_N\otimes A$.

\begin{definition}
  Suppose $H$ and $K$ are Hilbert spaces and $A$ is a unital self-adjoint subspace of $\sB(H)$. 
A  mapping $\rho:A\rightarrow \sB(K)$ is called positive if it maps positive elements to positive elements i.e. $\rho(a)\geq 0$ if $a\geq 0$. A mapping $\rho:A\rightarrow \sB(K)$ is called completely positive if the mapping $I_n\otimes \rho:M_n\otimes A\rightarrow M_n\otimes \sB(K)$ is positive for all $n\in \mathbb{N}$.
\end{definition}

\begin{definition}
Let $n$, $N$ and $M$ be given positive integers. An \df{hereditary polynomial} $p(x,y)$ (in two variables) of size $n$ and bi-degree at most $(M,N)$ in invertible variables $x_1, y_1, x_2, \text{ and } y_2$  such that $y_1$ and $y_2$ commute and $x_1$ and $x_2$ commute,
is a polynomial of the form
\begin{equation}\label{hereditarypolydegreetwo} 
p(x_1,y_1,x_2,y_2)=\sum_{\substack{\delta, \gamma =-M\\ \alpha, \beta=-N}}^{M,N}p_{\gamma, \alpha, \beta,\delta} y_2^{\gamma}y_1^\alpha x_1^\beta x_2^\delta. 
\end{equation}
Here the sum is finite and $p_{\gamma,\alpha,\beta,\delta}$ are $n\times n$ matrices over $\mathbb{C}$. Again, let $\mathcal{P}_n$ be the collection of 2-variable hereditary polynomials of size $n$ and let $\mathcal{P}=(\mathcal{P}_n)_n$ denote the collection of all hereditary polynomials.
\end{definition}

 Given a pair of commuting invertible  operators $T_1$ and $T_2$ on the Hilbert space $H$, let 
\begin{equation}\label{TSspan}
\mathcal{H}(T_1, T_2)=\text{span}\{T_2^{*\gamma}T_1^{*\alpha}T_1^\beta T_2^\delta:\ \gamma,\alpha, \beta,\delta \in \mathbb{Z}\}.
\end{equation}
 Note that $\mathcal{H}(T_1,T_2)$ is a unital self-adjoint subspace of $\sB(H)$. Recall that the Gelfand-Niamark-Segal construction realizes an abstract $C^*$-algebra as a subalgebra (unital and self-adjoint) of some $\sB(H)$. 

\begin{thm}[Stinespring]\label{paulsen}
Let $\mc{A}$ be a unital $C^*$-algebra and  $\phi:\mc{A}\rightarrow \sB(H)$ a linear map. If $\phi$ is completely positive, then there exists a Hilbert space $\mathcal{K}$, a unital $*$-homomorphism $\pi:\mc{A}\rightarrow B(\mc{K})$, and a bounded operator $V:H\rightarrow \mc{K}$ with $\|\phi(1)\|=\|V\|^2$ such that 
\[\phi(a)=V^*\pi(a)V.\] 
\end{thm}

We now present a version of the Arveson Extension Theorem for 2-tuples of operators.
 
\begin{thm}[Arveson Extension Theorem]\label{Arvesonext}
Suppose that $T_1$ and $T_2$ are invertible operators on a Hilbert space $H$ and $S_1$ and $S_2$ are invertible operators on a Hilbert space $K$ . There is a Hilbert space $\mathcal{K}$, a representation $\pi:\sB(K)\rightarrow \sB(\mathcal{K})$, and an isometry $V:H\rightarrow K$ such that $VT_1^{\beta}T_2^{\gamma}=\pi(J_1)^{\beta}\pi(J_2)^\gamma V$ for all $\beta, \gamma\in \mathbb{Z}$ if and only if the mapping $\rho:\mathcal{H}(J_1,\ J_2)\rightarrow \mathcal{H}(T_1,T_2)$ is completely positive. 
\end{thm}
\begin{proof}Suppose $\rho:\mathcal{H}(J_1,\ J_2)\rightarrow \mathcal{H}(T_1,\ T_2)$ determined by $\rho(J_2^{*\gamma}J_1^{*\alpha}J_1^\beta J_2^\delta)=T_2^{*\gamma}T_1^{*\alpha}T_1^\beta T_2^\delta$ is well defined and completely positive. In this case,  by Theorem \ref{paulsen}, there is a Hilbert space $\mathcal{K}$, a representation $\pi:\sB(H)\rightarrow \sB(\mathcal{K})$ and an isometry $V:H\rightarrow \mathcal{K}$ such that 
\[V^* \pi(J_2^{*\gamma}J_1^{*\alpha}J_1^\beta J_2^\delta)V= \rho(J_2^{*\gamma}J_1^{*\alpha}J_1^\beta J_2^\delta)=T_2^{*\gamma}T_1^{*\alpha}T_1^\beta T_2^\delta.\]
Since $\pi$ is an algebraic homomorphism which preserves involultions, 
\begin{equation}\label{eq: powerrule}
V^*\pi(J_2)^{*\gamma}\pi(J_1)^{*\alpha}\pi(J_1)^\beta\pi(J_2)^\delta V= T_2^{*\gamma}T_1^{*\alpha}T_1^\beta T_2^\delta.
\end{equation}
For each $\gamma, \beta\in \mathbb{Z}$, 
\begin{align*}
V^*\pi(J_2)^{*\gamma}\pi(J_1)^{*\beta}\pi(J_1)^\beta\pi(J_2)^\gamma V&= T_2^{*\gamma}T_1^{*\beta}T_1^\beta T_2^\gamma\\
&=V^*\pi(J_2)^{*\gamma}\pi(J_1)^{*\beta}VV^*\pi(J_1)^\beta\pi(J_2)^\gamma V
\end{align*}
by Equation \eqref{eq: powerrule}.
Hence
\[V^*\pi(J_2)^{*\gamma}\pi(J_1)^{*\beta}\pi(J_1)^\beta\pi(J_2)^\gamma V-V^*\pi(J_2)^{*\gamma}\pi(J_1)^{*\beta}VV^*\pi(J_1)^\beta\pi(J_2)^\gamma V=0.\]
Since $I-VV^*$ is a projection and hence idempotent,
\[V^*\pi(J_2)^{*\gamma}\pi(J_1)^{*\beta}(I-VV^*)^2\pi(J_1)^\beta\pi(J_2)^\gamma V=0.\]
Therefore
\[(I-VV^*)\pi(J_1)^\beta\pi(J_2)^\gamma V=0.\]
Consequently
\[\pi(J_1)^\beta\pi(J_2)^\gamma V=VV^*\pi(J_1)^\beta\pi(J_2)^\gamma V.\]
Again by Equation \eqref{eq: powerrule},
\[VT_1^{\beta}T_2^{\gamma}=\pi(J_1)^{\beta}\pi(J_2)^\gamma V.\]

 Since the converse is not needed for any of our theorems, we omit the straightforward proof.
\end{proof}
In \cite{McCullough}, a strong variant of Theorem \ref{Arvesonext} was proven using Agler's symmetrization technique. 
\begin{definition}
Given a two-variable hereditary polynomial $p(x_1,x_2,y_1,y_2)$ as in Equation \ref{hereditarypolydegreetwo},  
 define its \df{symmetrization} $p^s$ by
\begin{equation}\label{twovarsymmetrization}
p^s=\sum p_{\beta,\alpha,\alpha,\beta}y_2^\beta y_1^\alpha x_1^\alpha x_2^\beta.
\end{equation}
\end{definition}
Similarly, let
\begin{equation}\label{TSsymspan}
\mc{H}_s(T_1,\ T_2)=\text{span}\{T_2^{*\beta}T_1^{*\alpha}T_1^\alpha T_2^\beta \ :\ \alpha,\beta\in \mathbb{Z}\}. 
\end{equation}
In order to prove a strong variant of Theorem (\ref{Arvesonext}) we will need several lemmas. They are presented below.

\begin{definition}[Pairwise Rotationally Symmetric] A pair of operators $S_1$ and $S_2$ is  \df{pairwise rotationally symmetric} if for all $t\in \mathbb{R}^2$, $t=(t_1, t_2)$,  there exists a unitary operator $U_t$ such that 
\[e^{it_1}S_1=U_t^*S_1U_t \text{ and }e^{it_2}S_2=U_t^*S_2U_t.\]
\end{definition}

\begin{example}
 \label{pair-sym-ex}
Define on  $\L^2(\mathbb{T}^2)$ the operators 
\begin{equation}\label{eq:Z_1}
Z_1:\L^2(\mathbb{T}^2)\rightarrow \L^2(\mathbb{T}^2) \ \ \ \ Z_1f(z_1,z_2)=z_1f(z_1,z_2)
\end{equation}
and 
\begin{equation}\label{eq:Z_2}
Z_2:\L^2(\mathbb{T}^2)\rightarrow \L^2(\mathbb{T}^2) \ \ \ \ Z_2f(z_1,z_2)=z_2f(z_1,z_2).
\end{equation}
   Given $t$, define $U_t$ on $\L^2(\mathbb{T}^2)$ by $U_t f(\zeta_1,\zeta_2) = f(\exp(it_1)\zeta_1,\exp(it_2)\zeta_2)$. A calculation shows 
   $U_t Z_j = \exp(it_j) Z_j U_t$.  Hence the pair $(Z_1,Z_2)$ is pairwise rotationally symmetric.
\end{example}

\begin{lemma}
 \label{lem:makepairwise}
    If $S_1$ and $S_2$ are pairwise rotationally symmetric operators and $T_1$ and $T_2$ are operators on a common Hilbert space, then $\widetilde{T}_1=T_1\otimes S_1$ and $\widetilde{T}_2=T_2\otimes S_2$ are pairwise rotationally symmetric.
\end{lemma}
\begin{proof}
Since $S_1$ and $S_2$ are pairwise rotationally symmetric, for each $t=(t_1,t_2)\in\mathbb R^2$  there exists a unitary operator $U_t$ such that 
\[e^{it_1}S_1=U_t^*S_1U_t \quad\text{ and }\quad e^{it_2}S_2=U_t^*S_2U_t.\]
Since $e^{it_1}\tw{T_1}=T_1\otimes e^{it_1}S_1$ and $e^{it_2}\tw{T_2}=T_2\otimes e^{it_2}S_2$, to see that $\tw{T_1}$ and $\tw{T_2}$ are pairwise rotationally symmetric, consider the operators $\tw{U_t}=(I\otimes U_t)$. 
\end{proof}
\begin{lemma}\label{twiddlemaplemma}
If $J_1$ and $J_2$ are pairwise rotationally symmetric, $q\in \mathcal{P}$and $q(J_1,J_2)\succeq 0,$ then $q^s(J_1,J_2)\succeq~0$. 

Let $T_1$ and $T_2$ be given invertible operators on the Hilbert space $H$ and let $W:H\rightarrow H\otimes \L(\mathbb{T}^2)$ denote the isometry $Wh=h\otimes 1$. If $P\in \mathcal{P}_n$, then 
\[P^s(T_2^*,T_1^*,T_1,T_2)=(I_n\otimes W)^*P(\tw{T_2}^*, \tw{T_1}^*,\tw{T_1},\tw{T_2})(I_n\otimes W).\]
\end{lemma}

We will occasionally use the notation $p(T^*,T)$ for $p(T_2^*,T_1^*,T_1,T_2)$.

\begin{proof}
Let $n$ denote the size of $q$ (i.e. $q\in \mc{P}_n$).  For each $t=(t_1,t_2)\in \mathbb{R}^2$ there is a unitary operator $U_t$ such that 
\[e^{it_1}J_1=U_t^*J_1U_t \quad\text{ and }\quad e^{it_2}J_2=U_t^*J_2U_t\] 
by a combination of Lemma \ref{lem:makepairwise} and Example \ref{pair-sym-ex}. Hence
\[U_t^*J_2J_1U_t=U_t^*J_2U_tU_t^*J_1U_t.\]
It follows that 
\[q(e^{-it_2}J_2^*, e^{-it_1}J_1^*, e^{it_1}J_1,e^{it_2}J_2)=(I_n\otimes U_t)^*q(J_2^*,J_1^*, J_1, J_2)(1\otimes U_t)\succeq 0.\]
Hence,
\[q^s(J_2^*,J_1^*, J_1, J_2)=\frac{1}{4\pi^2}\int_{0}^{2\pi}\int_{0}^{2\pi}q(e^{-it_2}J_2^*, e^{-it_1}J_1^*, e^{it_1}J_1,e^{it_2}J_2)\ dt\succeq 0.\]
To prove the second assertion, let $p\in \mathcal{P}_1$ and compute
\begin{align*}
\ip{p(\tw{T_2}^*,\tw{T_1}^*,\tw{T_1},\tw{T_2})Wh}{Wf}=&\ip{p(\tw{T_2}^*,\tw{T_1}^*,\tw{T_1},\tw{T_2})h\otimes 1}{f\otimes 1}\\
=&\ip{\sum_{\gamma, \alpha,\beta,\delta}p_{\gamma, \alpha,\beta,\delta}T_1^{\beta}T_2^{\delta}h \otimes e^{it_2\delta}e^{it_1\beta}}{T_1^{\alpha}T_2^{\gamma}f \otimes e^{it_2\gamma}e^{it_1\alpha}}\\
=&\ip{\sum p_{\beta,\alpha,\alpha,\beta}T_2^{*\beta}T_1^{*\alpha}T_1^\alpha T_2^\beta h}{f}\\
=&\ip{p^s(T_2^*,T_1^*,T_1,T_2)h}{f}.
\end{align*}
Applying this result entry-wise, we get the result for $P$.
\end{proof}

\begin{lemma}\label{twiddlemap}
Suppose $T_1$, $T_2$ are invertible operators on a Hilbert space $H$ and $J_1$ and $J_2$ are invertible operators on a Hilbert space $K$. If $J_1$ and $J_2$ are pairwise rotationally symmetric and the mapping $\rho:\mc{H}_s(J_1, J_2)\rightarrow \mc{H}_s(T_1,T_2)$ determined by $\rho(J_2^{\beta*}J_1^{\alpha*}J_1^\alpha J_2^\beta)=T_2^{\beta*}T_1^{\alpha*}T_1^\alpha T_2^\beta$ is (well defined and) completely positive, then the mapping $\tw{\rho}:\mc{H}(J_1,J_2)\rightarrow \mc{H}(\tw{T_1},\tw{T_2})$ determined by 
\[\tw{\rho}(J_2^{\gamma*}J_1^{\alpha*}J_1^{\beta}J_2^{\delta})=\tw{T_2}^{\gamma*}\tw{T_1}^{\alpha^*}\tw{T_1}^{\beta}\tw{T_2}^{\delta}\]
is also (well defined and) completely positive.
\end{lemma}

\begin{proof}Fix a positive integer $n$ and a $p\in \mathcal{P}_n$  and suppose $p(J^*,J)\succeq 0$. We are to show $p(\tw{T}^*,\tw{T})\succeq 0$. Given a pair of integers $(M,N)$ let $P$ denote the $(2M+1)\times (2M+1)$ matrix whose entries are the $(2N+1)\times(2N+1)$ matrices whose entries are $n\times n$ matrices,   
\begin{equation}\label{eq:twomatlevel}P=\left(\left(\left(I_n\otimes y_2^{j_2}\right)\left(I_n\otimes y_1^{j_1}\right)p(x,y)\left(I_n\otimes x_1^{k_1}\right)\left(I_n\otimes x_2^{k_2}\right)\right)_{j_1,k_1=-N}^N \right)_{j_2,k_2=-M}^M
\end{equation}
Thus $P(T^*,T)$ is an operator on $((\mathbb{C}^n\otimes H)\otimes \mathbb{C}^{2N+1})\otimes \mathbb{C}^{2M+1}$ and the entries of $P(T^*,T)$ are operators of $(\mathbb{C}^n\otimes H)\otimes \mathbb{C}^{2N+1}$ given by 
\begin{equation}\label{eq:Pentry} \left(\left(I_n\otimes T_2^{j_2}\right)\left(I_n\otimes T_1^{j_1}\right)p(T^*,T)\left(I_n\otimes T_1^{k_1}\right)\left(I_n\otimes T_2^{k_2}\right)\right)_{j_1,k_1=-N}^N.
\end{equation}
Note that $P(J^*,J)\succeq 0$ and thus, by Lemma \ref{twiddlemaplemma}, $P^s(J^*,J)\succeq 0$. Thus, by the hypotheses of this lemma, $P^s(T^*,T)\succeq 0$.
Let $\{e_1,\dots,e_n\}$ denote the standard basis for  $\mathbb C^n$.  Reusing notation, let $\{f_{-N},\ldots, f_0,\ldots , f_N\}$ and  $\{f_{-M},\ldots, f_0,\ldots , f_M\}$  denote the standard bases  for $\mathbb {C}^{2N+1}$ and  $\mathbb{C}^{2M+1}$ respectively. A generic vector in $\mathbb{C}^n\otimes H\otimes \mathbb{C}^{2N+1}\otimes \mathbb{C}^{2M+1}$, the space that $P(T^*,T)$ acts upon, has the representation
\[
 h = \sum h_{j,a,\alpha} \otimes e_j \otimes f_a \otimes f_\alpha.
\]
Let $p_{j,k}(\tw{T^*},\tw{T})$ denote the $j,k$-th entry of $p(\tw{T^*},\tw{T})$. Compute, using Lemma \ref{twiddlemaplemma},
\begin{equation}\label{densesums}
 \begin{split}
 0 \le &\ip{P^s(T_2^*, T_1^*, T_1, T_2)h}{h}\\
    &=\ip{P(\tw{T}_2^*, \tw{T}_1^*,\tw{T}_1, \tw{T}_2)h\otimes 1}{h\otimes 1}\\
    & =  \sum_{a,b,\alpha,\beta} \sum_{j,k} \ip{\tw{T}_1^{*b} \tw{T}_2^{*\beta} \, p_{j,k}(\tw{T}^*,\tw{T})\, \tw{T}_1^{a} \tw{T}_2^{\alpha}\, h_{j,a,\alpha}\otimes 1}{h_{k,b,\beta}\otimes 1}. \\
  &  = \sum_{a,b,\alpha,\beta}  \sum_{j,k} \langle  p_{j,k}(\tw{T}^*,\tw{T})\,  T_2^{\alpha} T_1^a h_{j,a,b}  \otimes z_1^a z_2^{\alpha}, \, T_2^{\beta} T_1^b h_{k,\alpha,\beta} \otimes z_1^b z_2^\beta \rangle \\
 & = \sum_{j,k} \langle p_{j,k}(\tw{T}^*,\tw{T})[ \sum_{a,\alpha} T_2^{\alpha} T_1^a h_{j,a,b}  \otimes z_1^a z_2^{\alpha}], \,  [\sum_{b,\beta}  T_2^{\beta}T_1^b h_{k,b,\beta}\otimes z_1^b z_2^{\beta}] \rangle \\
  & = \langle p(\tw{T}^*,\tw{T}) g,\, g\rangle,
\end{split}
\end{equation}
 where 
\[
  g = \sum_{j=1}^n \sum_{a=-N}^N \sum_{\alpha=-M}^M T_2^{\alpha} T_1^a h_{j,a,b}  \otimes z_1^a z_2^{\alpha}. 
\]
 Since $T_1$ and $T_2$ are invertible, given vectors $g_{j,a,b} \in H$, there exists vectors $h_{j,a,b}$ such that 
\[
 g = \sum_{j=1}^n \sum_{a=-N}^N \sum_{\alpha=-M}^M  g_{j,a,b}  \otimes z_1^a z_2^{\alpha}. 
\]
 Finally, since vectors of the form $g$ are dense in $H\otimes L^2(\mathbb T^2)$, it follows that $p(\tw{T}^*,\tw{T})\succeq 0$; i.e., that map $\tw{\rho}$ is completely positive.
\end{proof}

\begin{lemma}\label{taumap}
Suppose $T_1$ and $T_2$ are invertible operators in $\sB(H).$  If $p\in \mathcal{P}$ and\\ $p(\tw{T_2}^*, \tw{T_1}^*, \tw{T_1}, \tw{T_2})\geq0$, then $p({T_2}^*, {T_1}^*, T_1, T_2)\geq 0$. In particular the mapping \[\tau:p(\tw{T_2}^*, \tw{T_1}^*, \tw{T_1}, \tw{T_2})\mapsto p({T_2}^*, {T_1}^*, T_1, T_2)\] is well defined. 
\end{lemma}
\begin{proof}
Let 
\[D_{NM}=\frac{1}{\sqrt{2N+1}}\frac{1}{\sqrt{2M+1}}\sum_{j=-N}^N\sum_{ k=-M}^{M}e^{ijt_1}e^{ikt_2} \in L^2(\mathbb{T}^2).\]
If $f,h\in H,$ then for $\alpha, \beta, \gamma, \delta \in \mathbb{Z}$, 
\begin{align*}
&\ip{\tw{T_2}^{*\gamma}\tw{T_1}^{*\beta}\tw{T_1}^{\alpha}\tw{T_2}^\delta h\otimes D_{N,M}}{f\otimes D_{N,M}}\\&=\ip{\tw{T_1}^\alpha\tw{T_2}^\delta  h\otimes D_{N,M}}{\tw{T_1}^\beta \tw{T_2}^\gamma f\otimes D_{N,M}}\\
&=\ip{T_1^\alpha T_2^\delta h}{T_1^\beta T_2^\gamma f}\ip{z_1^\alpha z_2^\delta D_{N,M}}{z_1^\beta z_2^\gamma D_{N,M}}\\
&=\ip{T_1^\alpha T_2^\delta h}{T_1^\beta T_2^\gamma f}\left(\frac{1}{(2M+1)(2N+1)}\right)\ip{\sum_{j=-N+|\alpha-\beta|}^{N+|\alpha-\beta|}\sum_{k=-M+|\gamma-\delta|}^{M+|\gamma-\delta|}e^{ijt_1}e^{ikt_2}}{\sum_{j=-N}^{N}\sum_{k=-M}^{M}e^{ijt_1}e^{ikt_2}}\\
&=\ip{T_1^\alpha T_2^\delta h}{T_1^\beta T_2^\gamma f}\left(\frac{2N+1-|\alpha-\beta|}{2N+1}\right)\left(\frac{2M+1-|\gamma-\delta|}{2M+1}\right).
\end{align*}
Thus if $p\in \mathcal{P}_1$, 
\[\lim_{N\rightarrow \infty}\lim_{M\rightarrow \infty}\ip{p(\tw{T}_2^*, \tw{T}_1^*,\tw{T}_1, \tw{T}_2)h\otimes D_{N,M}}{f\otimes D_{N,M}}=\ip{p({T}_2^*, {T}_1^*,{T}_1, {T}_2)h}{f}.\]
Hence if $p(\tw{T}_2^*, \tw{T}_1^*,\tw{T}_1, \tw{T}_2)\succeq 0,$ then $p({T}_2^*, {T}_1^*,{T}_1, {T}_2)\succeq 0$ as well. The case for square matrices is easily established. 
\end{proof}
\begin{prop}\label{Arvesonextsym}
Suppose $T_1$ and $T_2$ are invertible operators on a Hilbert space $H$, and $J_1$ and $J_2$ are invertible operators on a Hilbert space $K$. If $J_1$ and $J_2$ are pairwise rotationally symmetric and the mapping $\rho:\mc{H}_s(J_1, J_2)\rightarrow \mc{H}_s(T_1, T_2)$
determined by $\rho(J_2^{*\beta}J_1^{*\alpha}J_1^\alpha J_2^\beta)=T_2^{*\beta}T_1^{*\alpha}T_1^\alpha T_2^\beta$ is well defined and completely positive, then there is a Hilbert space $\mc{K}$, a representation $\pi:\sB(K)\rightarrow B(\mc{K})$, and a isometry $V$ such that $VT_2^{m}T_1^{n}=\pi(J_1)^{n}\pi(J_2)^{m}V$ for $m,n \in \mathbb{Z}$.
\end{prop}

\begin{proof}
The mapping  $\tau:\mc{H}(\tw{T_1},\tw{T_2})\rightarrow \mc{H}(T_1,T_2)$ as described in Lemma \ref{taumap}, is well defined and completely positive. The mapping $\tw{\rho}:\mc{H}(J_1,J_2)\rightarrow \mc{H}(\tw{T_2},\tw{T_2})$ as described in \ref{twiddlemap} is also well defined and completely positive. 
Their composition 
\[\rho=\tau\circ \tw{\rho}\]
 is well defined and completely positive. The proposition now follows from Theorem \ref{Arvesonext}.
\end{proof}

Fix $c,d>0$ and  define, for $0\le i+j\le 2$ (here $i,j$ are non-negative integers),  the $3\times 3$ matrices  $B_{i,j}$ by
\begin{equation}
 \label{defineBij}
 I+\sum_{0< i+j\le 2}  B_{i,j} \alpha^i \beta^j = \begin{pmatrix} 1 & \alpha \, c &  \beta\, d \\ \alpha \, c & 1+\alpha^2 \, c^2 & \alpha \beta\, cd \\ \beta \, d & \alpha\beta \, cd & 1+\beta^2 \, d^2 \end{pmatrix},
\end{equation}
 and $B_{0,0} = I -\frac{1}{c^2}B_{2,0} -\frac{1}{d^2} B_{0,2}$.
Define, 
\begin{equation}\label{Jform}
\J_1=\begin{pmatrix}U_1 & c U_1& 0\\ 0 & U_1 & 0\\ 0 & 0 & U_1\end{pmatrix}
\J_2=\begin{pmatrix}U_2 & 0&d U_2\\0 & U_2 & 0\\ 0 & 0 & U_2\end{pmatrix},
\end{equation}
where $U_1=Z_1$ and $U_2=Z_2$, the pairwise rotationally symmetric operators in Example \ref{pair-sym-ex}. We note $\J_1$ and $\J_2$ are pairwise rotationally symmetric via Lemma \ref{lem:makepairwise}. It is clear that from the calculation done in Equation \eqref{eq:J_1J_2} that $\J=(\J_1,\J_2)\in \mathfrak{J}_{c,d}$ and
\begin{equation} \label{eq:Q(J_1,J_2)}
 Q_{\J}(\alpha,\beta) =  \big (I+\sum_{0<i+j\le 2}  B_{i,j} \alpha^i \beta^j \, \big )\otimes I.
\end{equation}
 In particular $B_{i,j}(\J) = B_{i,j}\otimes I$ and we  define  $B_{0,0}(\J) = B_{0,0}\otimes I$. 

\begin{lemma}\label{factorgivespos} 
 If  $T=(T_1, T_2)$ is in the class $\mathfrak{F}_{c,d}$, and
\footnotesize
\[\hat{Q}_{T}(\alpha,\beta)=Q_T(\alpha,\beta)-\frac{1}{c^2}B_{2,0}(T)-\frac{1}{d^2}B_{0,2}(T)\succeq 0\]
\normalsize
factors in the form, 
\begin{equation}
 \label{eq:factorgivespos}
   \hat{Q}_{T}(\alpha,\beta)=(V_0+\alpha V_1 + \beta V_2)^*(V_0+\alpha V_1 + \beta V_2),
\end{equation}
then the map $\rho(\J_2^{*\beta}\J_1^{*\alpha}\J_1^\alpha \J_2^\beta)=T_2^{*\beta}T_1^{*\alpha}T_1^{\alpha}T_2^{\beta}$ is well defined and completely positive. 
\end{lemma}

\begin{proof}
 Suppose the 2-tuple $T=(T_1, T_2)$ is in the class $\mathfrak{F}_{c,d}$ and for notational convenience let 
\[{B}_{0,0}(T)=I-\frac{1}{c^2}{B}_{2,0}(T)-\frac{1}{d^2}{B}_{0,2}(T)=I-\frac{1}{c^2}\Btwo{T_1}-\frac{1}{d^2}\Btwo{T_2}.\]
Note that 
\[{B}_{0,0}(T)\succeq 0\]
since $Q_{T}(\alpha,\beta)\succeq 0$ for $\alpha=\beta=0.$
The spaces $\mc{H}_s(\J_1,\J_2)$ and $\mc{H}_s(T_1,T_2)$ are spanned by 
\[\{{B}_{0,0}(\J),\ {B}_{1,0}(\J),\ {B}_{0,1}(\J),\ {B}_{1,1}(\J),\ {B}_{2,0}(\J),\ {B}_{0,2}(\J)\}\]
and
\[\{{B}_{0,0}(T),\ {B}_{1,0}(T),\ {B}_{0,1}(T),\ {B}_{1,1}(T),\ {B}_{2,0}(T),\ {B}_{0,2}(T)\}\]
respectively. For positive integers $n$, let $M_n$ denote the  $n\times n$ matrices. The elements $X\in M_n\otimes \mc{H}_s(\J_1,\J_2)$  have 
 the form
\[
 X= \sum_{0\le i+j\le 2} X_{i,j}\otimes B_{i,j}(\J).
\]
Equivalently,
\[
X\cong\begin{pmatrix}X_{0,0}& cX_{1,0} & dX_{0,1}\\ cX_{1,0} & c^2X_{2,0} & cdX_{1,1}\\ dX_{0,1} & cdX_{1,1} & d^2X_{0,2}\end{pmatrix}\otimes I.
\]
If $X\succeq 0$, then each $X_{i,j}$ is self-adjoint. Further $X\succeq 0$ if and only if 
\[Y=\begin{pmatrix}X_{0,0}& X_{1,0} & X_{0,1}\\ X_{1,0} & X_{2,0} & X_{1,1}\\ X_{0,1} & X_{1,1} & X_{0,2}\end{pmatrix}
\] 
is as well. In this case, there exists $3n\times n$ matrices $Y_0,Y_1,Y_2$ such that
 \[\begin{pmatrix}X_{0,0}& X_{1,0} & X_{0,1}\\ X_{1,0} & X_{2,0} & X_{1,1}\\ X_{0,1} & X_{1,1} & X_{0,2}\end{pmatrix}=\begin{pmatrix}Y_0^*\\ Y_1^*\\ Y_2\end{pmatrix}\begin{pmatrix}Y_0&Y_1&Y_2\end{pmatrix}.
\] 
Using the  factorization \eqref{eq:factorgivespos}, 
\begin{equation}
\begin{aligned}
1_m\otimes \rho(X)&=\sum X_{i,j}\otimes B_{i,j}(T)\\
&=X_{0,0}\otimes V_0^*V_0+X_{1,0}\otimes (V_0^*V_1+V_1^*V_0)+X_{0,1}\otimes(V_0^*V_2+V_2^*V_0)\\
&\hspace{.5cm}+X_{1,1}\otimes(V_1^*V_2+V_2^*V_1)+X_{2,0}\otimes(V_1^*V_1)+X_{0,2}\otimes(V_2^*V_2)\\
&=(Y_0\otimes V_0+Y_1\otimes V_1+Y_2\otimes V_2)^*(Y_0\otimes V_0+Y_1\otimes V_1+Y_2\otimes V_2).
\end{aligned}
\end{equation}
Since the right hand side is evidently positive, the map $\rho$ is completely positive.
\end{proof}

By Proposition \ref{Arvesonextsym} and Lemma \ref{factorgivespos} since $\J_1$ and $\J_2$ are pairwise rotationally symmetric, we have shown a factorization \eqref{eq:factorgivespos} implies there is a representation $\pi$ such that the 2-tuple $T$ lifts to the 2-tuple $\pi(\J)$. It remains to show that 
 any representation applied to $\J=(\J_1,\J_2)$ produces a 2-tuple of the same form.

\begin{lemma}\label{properform}
Let $E$ be the Hilbert space that $\J_1$ and $\J_2$ act upon. If $\tw{E}$ is also a Hilbert space and $\pi:B(E)\rightarrow B(\tw{E})$ is a unital $*$-representation, then 
$J_1=\pi(\J_1)$ and $J_2=\pi(\J_2)$ have, up to unitary equivalence, the same form as $\J_1$ and $\J_2$ given by Equation \eqref{eq:Jforms} and in particular  are in the class $\mathfrak{J}_{c,d}$.
\end{lemma}
\begin{proof}The proof proceeds much in the same way as it does in \cite{McCullough} but with some minor differences. 
 The following relations are evident. 
\begin{enumerate}[i)]
\item $\J=\mc{W}_i+\mc{N}_i$ where $\mc{W}_i$ is unitary, $\mc{N}_i^2=0$  for $i=1,2$ .
\item $\mc{W}_i\mc{N}_i=  \mc{N}_i\mc{W}_i$  for $i=1,2$ .
\item $\mc{N}_1\mc{N}_1^*=\mc{N}_2\mc{N}_2^*$. 
\item $\mc{N}_1\mc{N}_1^*+\mc{N}_1^*\mc{N}_1+\mc{N}_2^*\mc{N}_2=1$.
\item $\mc{N}_i\mc{N}_j=0$ for $i,j=1,2$. 
\item $\mc{N}_i\mc{N}^*_j=0$ for $i,j=1,2$.
\end{enumerate}
 From these relations, 
\[\mc{N}_1^*\mc{N}_1,\] 
\[\mc{N}_2^*\mc{N}_2,\]
\[\mc{N}_1\mc{N}_1^*=\mc{N}_2\mc{N}_2^*\]
 are pairwise orthogonal projections. Let $J_i=\pi(\J_i),$  $N_i=\pi(\mc{N}_i)$, and $W_i=\pi(\mc{W}_i)$ for $i=1,2$.
These must satisfy the same algebraic relations, i.e.
\begin{enumerate}[i)]
\item $J={W}_i+{N}_i$ where ${W}_i$ is unitary, ${N}_i^2=0$  for $i=1,2$ .
\item ${W}_i{N}_i=  {N}_i{W}_i$  for $i=1,2$ .
\item ${N}_1{N}_1^*={N}_2{N}_2^*$. 
\item ${N}_1{N}_1^*+{N}_1^*{N}_1+{N}_2^*{N}_2=1$.
\item ${N}_i{N}_j=0$ for $i,j=1,2$. 
\item ${N}_i{N}^*_j=0$ for $i,j=1,2$.
\end{enumerate}
 From these relations, 
\[{N}_1^*{N}_1,\] 
\[{N}_2^*{N}_2,\]
\[{N}_1{N}_1^*={N}_2{N}_2^*\]
are pairwise orthogonal projections on $\tw{E}$. For instance, 
\[N_1^*N_1=N_1^*({N}_1^*{N}_1+ {N}_2^*{N}_2+{N}_1{N}_1^*)N_1=(N_1^*N_1)^2.\]
Now decompose the space $H$ as $H=\ran({N}_1{N}_1^*)\oplus\ran({N}_1^*{N}_1)\oplus \ran({N}_2^*{N}_2)$.The mappings $N_j$ are unitary maps $Q_j$ from the range of $N_j^*$ to the range of $N_j$. Hence, with respect to the orthogonal decomposition
 of $H$ as $H=\ran({N}_1{N}_1^*)\oplus\ran({N}_1^*{N}_1)\oplus \ran({N}_2^*{N}_2)$,
\[
 N_1 = \begin{pmatrix} 0&Q_1 & 0\\ 0 & 0& 0 \\ 0&0&0\end{pmatrix}
\]
 and likewise,
\[
 N_2 = \begin{pmatrix} 0& 0 & Q_2\\ 0 & 0& 0 \\ 0&0&0\end{pmatrix}.
\]
 Thus, up to unitary equivalence, it may be assumed that $Q_j=I$ (and each of the summands in the direct sum decomposition is the same Hilbert space).
Write 
\[W_1=\begin{pmatrix}A_1&B_1 &C_1\\ D_1& E_1 & F_1 \\ G_1& H_1& J_1\end{pmatrix}\]
for some $A_1$, $B_1$, $C_1$, $D_1$, $E_1$, $F_1$, $G_1$, $H_1$, and $J_1$ operators. 
Since $W_1N_1=N_1W_1$, 
\[W_1N_1=\begin{pmatrix}A_1&B_1 &C_1\\ D_1& E_1 & F_1 \\ G_1& H_1& J_1\end{pmatrix}\begin{pmatrix}0& I & 0\\ 0 & 0 & 0\\0 & 0 & 0\end{pmatrix}=\begin{pmatrix}0&A_1 &0\\ 0& D_1 &0 \\0& G_1& 0\end{pmatrix},\]
and
\[N_1W_1=\begin{pmatrix}0& I & 0\\ 0 & 0 & 0\\0 & 0 & 0\end{pmatrix}\begin{pmatrix}A_1&B_1 &C_1\\ D_1& E_1 & F_1 \\ G_1& H_1& J_1\end{pmatrix}=\begin{pmatrix} D_1& E_1 & F_1\\ 0 & 0 & 0\\0 & 0 & 0\end{pmatrix},\]
we conclude
\[A_1=E_1\]
and
\[D_1=F_1=G_1=0.\]
Similarly, since $W_1N_2=N_2W_1$, 
\[A_1=J_1\]  and 
\[H_1=0.\]
Hence 
\[W_1=\begin{pmatrix}A_1&B_1 &C_1\\ 0& A_1 & 0 \\ 0& 0& A_1\end{pmatrix}.\]
Since $W_1$ is a unitary operator,  
\[W_1W_1^*=\begin{pmatrix}A_1&B_1 &C_1\\ 0& A_1 & 0 \\ 0& 0& A_1\end{pmatrix}\begin{pmatrix}A_1^*&0 &0\\B_1^*& A_1^* & 0 \\ C_1^*& 0& A_1^*\end{pmatrix}=\begin{pmatrix}I&0 &0\\ 0& I & 0 \\ 0& 0& I\end{pmatrix},\]
where $I$ is the identity operator. Hence,
\[A_1A_1^*+B_1B_1^*+C_1C_1^*=I,\]
\[A_1A_1^*=I,\]
\[A_1B_1^*=0,\]
\[A_1C_1^*=0.\]
Note that the {\it first two} relations above show that $B_1=C_1=0$ and $A_1^*$ is an isometry. Hence $W$ is diagonal with $A_1$ down the diagonal. Since $W$ is unitary, $A_1$ is unitary. It follows that 
\[W_1=\begin{pmatrix}U_1& 0 & 0\\ 0 & U_1 & 0\\0 & 0 & U_1\end{pmatrix},\]
where $U_1$ is a unitary operator. A similar argument shows that
\[ W_2=\begin{pmatrix}U_2& 0 &0\\ 0 & U_2 & 0\\0 & 0 & U_2\end{pmatrix},\]
where $U_2$ is a unitary operator. Since $[W_1,W_2]=0,$ it follows  that $[U_1,U_2]=0$. 
 Hence, up to unitary equivalence, the  $J_i$ have the form claimed.
\end{proof}
The forward direction of the main theorem has been established. We now need only to prove that lifting implies factorization of the associated operator pencil. However, this is readily established. If  $T=(T_1,T_2)$ lifts to $J=(J_1J_2)$, then 
\[V^* \big (Q_{J_1,J_2}(\alpha,\beta)-\frac{1}{c^2}B_{2,0}(J)-\frac{1}{d^2}B_{0,2}(J)\big )V=Q_{T_1,T_2}(\alpha,\beta)-\frac{1}{c^2}B_{2,0}(T)-\frac{1}{d^2}B_{0,2}(T).\]
Hence any factorization of 
\[\hat{Q}_J(\alpha,\beta)=(K_0+\alpha K_1+\beta K_2)^*(K_0+\alpha K_1+\beta K_2)\] gives the factorization of $\hat{Q}_T$ as
\[Q_T(\alpha,\beta)=V^*(K_0+\alpha K_1+\beta K_2)^*(K_0+\alpha K_1+\beta K_2)V.\]
Since  $\hat{Q}_J$ factors as
\[\hat{Q}_{J}(\alpha,\beta)=\left(\begin{pmatrix}1 \\ \a c \\ \b d\end{pmatrix}^*\begin{pmatrix}1& \a c & \b d  \end{pmatrix}\right)\otimes I,\]
 the conclusion follows. 

\section{The Counter-Example}
\label{sec:counter}
This section has three parts. Let $Q(\alpha,\beta)$ be an arbitrary two variable quadractic pencil
\begin{equation}\label{eq:Q} Q(\alpha,\beta)=I+\sum_{0< j+k\leq 2}\alpha^j\beta^kB_{j,k}\end{equation}
with coefficients $B_{j,k}$ operators on a separable Hilbert space $H$ such that
\begin{equation}\label{cdcondition}\hat{Q}(\alpha,\beta)=Q(\alpha,\beta)-\frac{1}{c^2}B_{2,0}-\frac{1}{d^2}B_{0,2}\succeq 0
\end{equation}
for all $(\alpha,\beta)\in \mathbb{R}^2$. In the first part we show by construction there exists a commuting 2-tuple of $3$-isometries $T\in \mathfrak{F}_{cd}$ such that $\hat{Q}_T$ factors if and only if $\hat{Q}$ factors. In the second part we show that given a positive integer $n$ and  positive map $\phi:\text{Sym}_3(\mathbb{C})\rightarrow M_n,$  if the canonical quadratic pencil it determines factors, then $\phi$  is completely positive. Hence, an example of Choi \cite{Choi} of a positive $\phi:\text{Sym}_3\rightarrow M_n$ which is not completely positive produces a quadratic two variable pencil which does not factor which in turn produces a counter-example to a natural generalization of the main lifting result of
 \cite{McCullough}. This counter-example is strengthened in the last part. 
\subsection{Constructing Three Isometries.}
\ \\
Let $F$ be a  vector space with  basis $\{f_j: j\in\mathbb Z\}$.  In particular,
 the set $\{f_j\otimes f_k : j,k \in\mathbb Z\}$ is a basis for the tensor product $F\otimes F$.
Define, on the algebraic tensor product $H\otimes (F\otimes F)$ the sesquilinear form
\[\ipsq{\hef}{\hefp}=\left\{ \begin{array}{ll} \ip{Q(j,k) h}{h^\prime}_H & \text{ if }j=j^\prime \text{ and }k=k^\prime \\ \ \ \ \ \ \ 0 &\text{ otherwise }\end{array}\right.,\]
and the linear maps
\begin{equation}\label{eq:T}T(\hef)=h\otimes f_{j+1}\otimes f_k\end{equation}
and
\begin{equation}\label{eq:S}S(\hef)=h\otimes f_j \otimes f_{k+1}.\end{equation}
Note that this sesquilinear form is positive semi-definite since $Q$ takes, by hypothesis, positive semi-definite values. Let $\mc{H}$ be the Hilbert space obtained from $H\otimes F\otimes F$ by modding out by the null vectors and forming the completion. We continue to denote the inner product on $\mc{H}$ by $\ipsq{\cdot}{\cdot}$ and let $h\otimes f_j\otimes f_k$ denote the equivalence
 class it represents in the quotient. We use freely the fact that $\mathcal D$, the linear span of $\{h\otimes f_j\otimes f_k: j,k\in\mathbb Z, \ \  h\in H\},$ is dense in $\mc{H}$.

\begin{prop}\label{shiftconstruct}
Given a 2-variable pencil in the form defined by \eqref{eq:Q}, if there exists  $c,d\in \mathbb{R}$ such that $c>0$, $d>0$ and 
\[Q(\a,\b)-\frac{1}{c^2}B_{2,0}-\frac{1}{d^2}B_{0,2}\succeq 0 \] for all $(\alpha,\beta)\in \mathbb{R}^2,$ then the operators $S$ and $T$ defined in \eqref{eq:T} and \eqref{eq:S} are well defined and extend to invertible bounded operators $\mc{H}$. Moreover $S$ and $T$ are 3-isometries  and
\[\ip{\hat{Q}_{T,S}(\alpha,\beta)(\hef)}{\gab}=\delta_{(j,k),(a,b)}\ip{\hat{Q}(\alpha+j, \beta+k)h}{h}_H,\]
 where $\delta$ is the Krocker delta function.
In particular, $(S,T)$ is in the class  $\mathfrak{F}_{c,d}.$ %
\end{prop}

\begin{proof}
Let $h=\hat{h}\otimes f_j\otimes f_k$ be an elementary tensor and 
 compute,
\begin{align*}
2(1+c^2)[h,h]-[Th,Th]&=\ip{(2Q(j,k)+2c^2Q(j,k)-Q(j+1,k))\hat{h}}{\hat{h}}\\
&=\ip{(Q(j,k)+2c^2Q(j,k)-B_{0,1}-kB_{1,1}-2jB_{2,0}-B_{2,0})\hat{h}}{\hat{h}}\\
&=\ip{(Q(j,k)+2c^2Q(j,k)-B_{0,1}-kB_{1,1}-2jB_{2,0}+B_{2,0}-2B_{2,0})\hat{h}}{\hat{h}}\\
&=\ip{(Q(j-1,k)+2c^2Q(j,k)-2B_{2,0})\hat{h}}{\hat{h}}.\\
\end{align*}
Since $Q(\a,\b)-\frac{1}{c^2}B_{2,0}-\frac{1}{d^2}B_{0,2}\succeq 0$ for all $(\alpha,\beta)\in \mathbb{R}^2,$  certainly $Q-\frac{1}{c^2}B_{2,0}\succeq 0$ and $Q\succeq 0$ for all $(\alpha,\beta)\in \mathbb{R}^2$.
Hence, $[2(1+c^2)[h,h]-[Th,Th]\geq0.$
Using orthogonality of the subspaces $\{h\otimes f_j\otimes f_k: h\in H\}$ for $j,k\in\mathbb Z$, it follows that for each $h\in H\otimes F\otimes F,$
\[
 2(1+c^2)[h,h]\ge [Th,Th].
\]
Thus $T$ is bounded on the algebraic tensor product and thus extends to a bounded operator, still denoted by $T$,  on $\mc{H}$ by continuity. A similar computation shows that $S$ is also bounded. 

It is straightforward to verify that
\[
 T^{*3}T^3 -3 T^{*2}T^2 + 3T^*T-I =0,
\]
 a condition well known to be equivalent to $T$ being a 3-isometry \cite{AglerStankusI,McCullough}.
 Likewise $S$ is a 3-isometry.  
Since $S$ and $T$ are 3-isometries there exist $\Bone{T}$, $\Bone{S}$, $\Btwo{T}$ and $\Btwo{S}$ such that for all natural numbers $m$ and $n$, 
\[S^{*m}S^m=I+m\Bone{S}+m^2\Btwo{S}\]
\[T^{*n}T^n=I+n\Bone{T}+n^2\Btwo{T}.\]
Define, $\tilde{B}_{1,0}=\Bone{T}$, $\tilde{B}_{0,1}=\Bone{S}$, $\tilde{B}_{2,0}=\Btwo{T}$, $\tilde{B}_{0,2}=\Btwo{S},$ and 
\begin{equation}\label{eq:B11}\tilde{B}_{1,1}= B_{1,1}\otimes I\otimes I. \end{equation}
 Direct computation shows
\begin{equation}\label{eq:BoneT}\ipsq{\Bone{T}(\hef)}{(\hab)}=\delta_{(j,k),(a,b)}\ip{(B_{1,0}+kB_{1,1}+2jB_{2,0})h}{h}_H,\end{equation}
\begin{equation}\label{eq:BoneS}\ipsq{\Bone{S}(\hef)}{\gab}=\delta_{(j,k),(a,b)}\ip{(B_{0,1}+jB_{1,1}+2kB_{0,2})h}{g}_H,\end{equation}
\begin{equation}\label{eq:BtwoT}\ipsq{\Btwo{T}(\hef)}{\gab}=\delta_{(j,k),(a,b)}\ip{B_{2,0}h}{g}_H,\end{equation}
\begin{equation}\label{eq:BtwoS}\ipsq{\Btwo{S}(\hef)}{\gab}=\delta_{(j,k),(a,b)}\ip{B_{0,2}h}{g}_H.\end{equation}
By the definition of $B_{1,1},$ 
\begin{equation}\label{eq:B11def}\ipsq{\tilde{B}_{1,1}(\hef)}{\gab}= \delta_{(j,k),(a,b)}\ip{B_{1,1} h}{g}.\end{equation}
From the above equations it follows that 
\begin{equation}\label{mixedterm}\ipsq{(S^{*m}\Bone{T}S^m)\hef}{\gab}=\delta_{(j,k),(a,b)}\ip{(B_{1,0}+(k+m)B_{1,1}+2jB_{2,0})h}{g}_H.\end{equation}
Likewise, 
\begin{equation}
 \label{eq:SBtwoT}
  \ipsq{S^{*m}\Btwo{T}S^m (\hef)}{\gab}=\delta_{(j,k),(a,b)} \ip{B_{2,0} h}{g}.
\end{equation}
Hence,  by equations \eqref{eq:BoneT},\eqref{eq:BoneS},\eqref{eq:BtwoT}, \eqref{eq:BtwoS}, \eqref{eq:B11def}, \eqref{mixedterm}, and \eqref{eq:SBtwoT}, 
\begin{align*}
&\ipsq{(S^{*m}T^{*n}T^nS^m)(\hef)}{\gab}\\
&=\ipsq{S^{*m}(1+nB_1(T)+n^2B_2(T))S^m (\hef)}{\gab}\\
&=\ipsq{I+mB_1(S)+m^2B_2(S)+nS^{*m}B_1(T)S^m+n^2B_2(T)(\hef)}{\gab}\\
&=\ipsq{(I+m\tilde{B}_{0,1}+n\tilde{B}_{1,0}+mn\tilde{B}_{1,1}+m^2\tilde{B}_{0,2}+n^2\tilde{B}_{2,0})(\hef)}{(\gab)}.
\end{align*}

We conclude,
\[Q_{T,S}(\alpha,\beta)=I+\alpha\tilde{B}_{1,0}+\beta\tilde{B}_{0,1}+\alpha\beta\tilde{B}_{1,1}+\alpha^2\tilde{B}_{2,0}+\beta^2\tilde{B}_{0,2}\]
The above equations give the following relationship
\[\ip{Q_{T,S}(\alpha,\beta)(\hef)}{\gab}=\delta_{(j,k),(a,b)} \ip{Q(\alpha+j, \beta+k)h}{g}_H\]
and 
\[\ip{\hat{Q}_{T,S}(\alpha,\beta)(\hef)}{\gab}= \delta_{(j,k),(a,b)} \ip{\hat{Q}(\alpha+j, \beta+k)h}{g}_H.\]

\end{proof}
\begin{prop}\label{bothfactor}
Let $Q(\alpha,\beta)$ be a quadratic pencil of the form \eqref{eq:Q} satisfying the positivity condition \eqref{cdcondition} and let $Q_{T,S}(\alpha,\beta)$ be the quadratic pencil for the 3-isometric 2-tuple $(T,S)\in \mathfrak{F}_{cd}$ constructed in Proposition (\ref{shiftconstruct}). The modified pencil $\hat{Q}(\alpha,\beta)$ factors if and only if the modified pencil $\hat{Q}_{T,S}(\alpha,\beta)$ factors.
\end{prop}

\begin{proof}
By the conclusion of Proposition (\ref{shiftconstruct}), 
\[\ip{\hat{Q}_{T,S}(\alpha,\beta)(\hef)}{(\heg)}=\delta_{(j,k),(a,b)} \ip{\hat{Q}(\alpha+j, \beta+k)h}{g}_H.\]
Suppose $\hat{Q}_{T,S}(\alpha,\beta)$ factors as 
\[\hat{Q}_{T,S}(\alpha,\beta)=(V_0+\alpha V_1+\beta V_2)^*(V_0+\alpha V_1+\beta V_2)\]
where $V_j$ are bounded operators from $\mc{H}$ into some auxiliary Hilbert space.
 Define  $U:H\rightarrow \mc{H}$ by
\begin{equation}
Uh=(h\otimes f_0 \otimes f_0).
\end{equation}
To verify that $U$ is an isometry,  note 
\[\|Uh\|=\|h\otimes f_0\otimes f_0\|=\|Q(0,0)^{\frac 12} h\|=\|h\|.\]
Now for all $g,h\in H$ 
\begin{align*}
\langle U^*(V_0+\alpha V_1+\beta V_2)^*&(V_0+\alpha V_1+\beta V_2)U h,g\rangle\\
=&\ip{U^*\hat{Q}_{T,S}(\alpha,\beta)U h}{g}\\
=&\ip{\hat{Q}_{T,S}(\alpha,\beta)Uh}{Ug}\\
=&\ip{\hat{Q}_{T,S}(\alpha,\beta)(h\otimes f_0\otimes f_0)}{(g\otimes f_0\otimes f_0)}\\
=&\ip{\hat{Q}(\alpha,\beta)h}{g}.
\end{align*}
 Thus, $\hat{Q}$ factors as
\[\hat{Q}(\alpha,\beta)=[(V_0+\alpha V_1+\beta V_2)U]^*\, [(V_0+\alpha V_1+\beta V_2)U].\]

Conversely, suppose that $\hat{Q}(\alpha,\beta)$ factors as 
\[\hat{Q}(\alpha,\beta)=(V_0+\alpha V_1+\beta V_2)^*(V_0+\alpha V_1+\beta V_2)\]
where the $V_j$ are bounded operators from $H$ into an auxiliary Hilbert space, which we label $K$ for convenience. 
 Let $\ell^2$ denote the Hilbert space $\ell^2(\mathbb Z)$ with the standard orthonormal basis $\{e_j: j\in\mathbb Z\}$
 and let $\mc{K}$ denote the Hilbert space tensor product $K\otimes (\ell^2 \otimes \ell^2)$. 
  Define, on the dense set $\mathcal D$,  equal to the span of elementary tensors $h\otimes f_j\otimes f_k$,  of $\mc{H}$ into $\mc{K}$ the linear maps,
\[
 \begin{split}
  W_0 (\sum h_{j,k}\otimes f_j\otimes f_k) = & \sum (V_0+jV_1+kV_2)h_{j,k}\otimes (e_j\otimes e_k) \\
  W_\ell (\sum h_{j,k} \otimes f_j\otimes f_k) = & \sum  V_\ell h_{j,k} \otimes (e_j\otimes e_k),
\end{split}
\]
  for $\ell=1,2.$   Since,
\[
 \begin{split}
 \langle W_0 (\sum h_{j,k} \otimes f_j\otimes f_k),& \, W_0 (\sum g_{a,b}\otimes f_a\otimes f_b)\rangle \\
  = & \sum_{j,k} \langle Q(j,k)h_{j,k}, \, g_{a,b} \rangle \\
  = & [ \sum h_{j,k} \otimes f_j\otimes f_k, \, \sum h_{a,b} \otimes f_a\otimes f_b ],
\end{split}
\]
 $W_0$ is an isometry on $\mathcal D$ and thus extends to an isometry, still denoted $W_0,$
 from $\mc{H}$ into $\mc{K}$.  Similarly,
\[
 \begin{split}
 \langle W_1 (\sum h_{j,k} \otimes f_j\otimes f_k),& \, W_1 (\sum h_{a,b}\otimes f_a\otimes f_b)\rangle \\
  = & \sum_{j,k} \langle V_1 h_{j,k},\, V_1 h_{j,k} \rangle  \\
  = & \sum_{j,k} \langle B_2(S^*,S) h_{j,k},\, h_{j,k} \rangle \\ 
  \le & c^2 \sum_{j,k} \langle Q(j,k)h_{j,k},\, h_{j,k} \rangle \\
  = & c^2 [\sum h_{j,k} \otimes f_j\otimes f_k,  \,  \sum h_{a,b}\otimes f_a\otimes f_b ].
\end{split}
\]
 Thus $W_1$ is bounded on $\mathcal D$ and thus extends to a bounded linear operator, still denoted
 $W_1$, from $\mc{H}$ to $\mc{K}$. Of course a similar statement holds for $W_2.$ 

Finally, 
\begin{align*}
 \langle (W_0+\alpha W_1+\beta W_2)^*&(W_0+\alpha W_1+\beta W_2)  (h_{j,k}\otimes f_j \otimes f_k),(g_{a,b}\otimes f_a \otimes f_b)\rangle \\
=&\ip{(W_0+\alpha W_1+\beta W_2)  (h_{jk}\otimes f_j\otimes f_k)}{(W_0+\alpha W_1+\beta W_2)  (h_{a,b}\otimes f_a\otimes f_b)}\\
=&\ip{(V_0+(\alpha+j)V_1+(\beta+k)V_2) h_{j,k}}{(V_0+(\alpha+j)V_1+(\beta+k)V_2) h_{a,b}}\\
=&\delta_{(j,k),(a,b)}\ip{\hat{Q}(\alpha+j,\beta+k)h_{j,k}}{h_{j,k}}\\
=&\ip{\hat{Q}_{T,S}(h_{j,k}\otimes f_j\otimes f_k)}{(h_{a,b}\otimes f_a\otimes f_b)}.
\end{align*}
Hence $\hat{Q}_{T,S}$ has the factorization $(W_0+\alpha W_1 + \beta W_2)^*(W_0+\alpha W_1 + \beta W_2)$.
\end{proof}

\subsection{A positive but not completely positive map.}
 In this section an example of Choi is used to produce a two-variable quadratic pencil
 which takes positive semidefinite values on $\mathbb R^2$, but does not factor. 
 In turn this pencil is used, in Proposition \ref{prop:counter}, to give a counter-example
 to a natural generalization of the main result of \cite{McCullough}.

\begin{definition}
 An \df{operator system} $S$ is a unital selfadjoint (vector) subspace of the bounded operators on a Hilbert space. 
Let $E_{i,j}$ denote the matrix units for $M_n$.  
The matrix 
\[ C_\phi = (\phi(E_{ij}))_{i,j} \in M_n\otimes S\] 
is the \df{Choi matrix} of the linear map $\phi:M_n\to S$. 
\end{definition}

 The following lemma can be found in \cite{Paulsen}

\begin{lemma}\label{lem:choi}
Let $S$ be an operator system. A map $\phi:M_n \rightarrow S$  is completely positive if and only if $C_\phi$ is positive semidefinite. 
\end{lemma}

 Recall the definitions of the $3\times 3$ matrices $B_{i,j}$ from equation \eqref{defineBij}. They form a basis for $\sym_3(\mathbb C)$. 

\begin{lemma}\label{factorlemma}
Suppose $S$  is an operator system and $\phi:\sym_3(\mathbb{C})\rightarrow S$ is a unital positive linear map.
If the canonical pencil
\[\hat{Q}_{\phi}(\a,\b)=\left[I+\sum_{0< i+j\leq 2} \alpha^i\beta^j\phi(B_{i,j})\right]-\frac{1}{c^2}\phi(B_{0,2})-\frac{1}{d^2}\phi(B_{2,0})=\sum_{0\leq j+k\leq 2}\alpha^j\beta^k \phi(B_{ij})\]
associated to $\phi$  factors as
\[\hat{Q}_{\phi}(\a,\b)=(V_0+\a V_1+\b V_2)^*(V_0+\a V_1+\b V_2),\]
where the $V_j$ are operators into an auxiliary space, then the map $\phi$ is completely positive.

Conversely, if the map $\phi$ is completely positive, then $\hat{Q}_\phi$  factors. 
\end{lemma}
\begin{proof}
Suppose that the canonical pencil factors as 
\[
 \hat{Q}_{\phi}(\a,\b) = (V_0 +\a V_1 +\b V_2)^* (V_0 +\a V_1 +\b V_2).
\]
An element $X\in M_n \otimes \sym(\mathbb{C})$ has the following form 
\[X\cong\begin{pmatrix}X_{0,0}& X_{1,0} & X_{0,1}\\ X_{1,0} & X_{2,0} & X_{1,1}\\ X_{0,1} & X_{1,1} & X_{0,2}\end{pmatrix}.\]
If $X\succeq 0,$ then each $X_{i,j}$ is self-adjoint and 
\[\begin{pmatrix}X_{0,0}& X_{1,0} & X_{0,1}\\ X_{1,0} & X_{2,0} & X_{1,1}\\ X_{0,1} & X_{1,1} & X_{0,2}\end{pmatrix}=\begin{pmatrix}Y_0\\Y_1\\Y_2\end{pmatrix}^*\begin{pmatrix}Y_0& Y_1& Y_2\end{pmatrix},
\] 
where the $Y_j$ are $3n\times n$ matrices. Thus, 
\begin{equation}
\begin{aligned}
(1_m\otimes \phi)(X)&=\sum X_{i,j}\otimes \phi(B_{i,j})\\
&=X_{0,0}\otimes V_0^*V_0+X_{1,0}\otimes (V_0^*V_1+V_1^*V_0)+X_{0,1}\otimes(V_0^*V_2+V_2^*V_0)\\
&\hspace{.5cm}+X_{1,1}\otimes(V_1^*V_2+V_2^*V_1)+X_{2,0}\otimes(V_1^*V_1)+X_{0,2}\otimes(V_2^*V_2)\\
&=(Y_0\otimes V_0+Y_1\otimes V_1+Y_2\otimes V_2)^*(Y_0\otimes V_0+Y_1\otimes V_1+Y_2\otimes V_2)\succeq 0.
\end{aligned}
\end{equation}
Hence $\phi$ is completely positive. 

We pause at this point to note some differences between the finite and infinite dimensional cases. There is a Hilbert space $\mathcal E$ such that $S\subset \sB(\mathcal E)$ and the $V_j$ map into an auxiliary Hilbert space $K$.  In fact,
\[V_j:\mathcal E \rightarrow \bigvee_{i=0}^2 \ran{V_i}.\]
 Thus, replacing $K$ by $\bigvee_{i=0}^2 \ran{V_i}$, it can be assumed that $V_j$ map into $\mathcal E^3$. Thus, if  $\mathcal E$ is finite dimensional, say $S\subset M_k$ (in which case there is no harm in assuming $S=M_k$),  then it can be assumed that $V_j$ map into an auxiliary space of dimension of at most $3k$. If $\mathcal E$ is an infinite dimensional space, then $\mathcal E^3$ can be identified with $\mathcal E$. 

Now suppose that the map $\phi:\sym_3{\mathbb{C}}\to S$ is completely positive and $S\subset \sB(\mathcal E)$.
 By Lemma \ref{lem:choi}, the Choi matrix $C_{\phi}$ is positive semidefinite and hence factors,
\[C_\phi=\begin{pmatrix}
\phi(E_{00})& \phi(E_{01})& \phi(E_{02})\\
\phi(E_{10})& \phi(E_{11})& \phi(E_{12})\\
\phi(E_{20})& \phi(E_{21})& \phi(E_{22})\\
\end{pmatrix}=\begin{pmatrix}V_0\\V_1\\V_2\end{pmatrix}^*\begin{pmatrix}V_0& V_1& V_2\end{pmatrix}
\] 
 where $V_j$ map $\mathcal E$ into an auxiliary Hilbert space.  To complete the proof, observe that
\[\hat{Q}_T(\a,\b)=(V_0+\alpha V_1+\beta V_2)^*(V_0+\alpha V_1+\beta V_2).\]
\end{proof}
We now present a map on $\sym_3(\mathbb{C})$ that is positive   but not completely positive. By Lemma \ref{factorlemma} this map produces a pencil that does not factor.

\begin{thm}[Choi]\label{notcpmap} There exists a positive linear map $\Phi:\sym_3(\mathbb{R})\rightarrow \sym_3(\mathbb{R})$ that does not admit an expression as $\Phi(A)=\sum V_i^\top A V_i$ with $3\times 3$ matrices $V_i$.  The map
\[(\alpha_{jk})_{jk}\mapsto 2\begin{pmatrix}\alpha_{11}+\alpha_{22} & 0 & 0 \\ 0 &\alpha_{22}+\alpha_{33}& 0 \\ 0 &0 & \alpha_{33}+\alpha_{11}\end{pmatrix}-(\alpha_{jk})_{jk}\]
is such an example.
\end{thm}

Choi's map  is not unital, since it sends the $I$ to $3I$. 
We correct this defect by multiplying by a positive scalar.

We will show that a variation of this map is not completely positive. 
\begin{prop}\label{nofactor} The unital positve map $\Phi:\sym_{3}(\mathbb{C})\rightarrow \sym_{3}(\mathbb{C})$ 
 given by 
\begin{equation}\label{Phimap}(\alpha_{jk})_{jk}\mapsto \frac{2}{3}\begin{pmatrix}\alpha_{11}+\alpha_{22} & 0 & 0 \\ 0 &\alpha_{22}+\alpha_{33}& 0 \\ 0 &0 & \alpha_{33}+\alpha_{11}\end{pmatrix}-\frac{1}{3}(\alpha_{jk})_{jk}\ \ \ \ \ \ \ \alpha_{jk}\in \mathbb{C}
\end{equation}
 is not completely positive. 
\end{prop}
\begin{proof}
For a matrix  $A$, let $\bar{A}$ denote the matrix whose entries are the conjugates of the entries of $A$. The notation $A^*$ and $A^\top$ will denote the conjugate transpose and transpose of $A$ respectively. Now suppose that $\Phi$ is completely positive and thus extends, via Arveson's extention theorem \cite{Paulsen}, to a completely positive map also denoted by $\Phi$ from $M_3(\mathbb C)$ to  $M_3(\mathbb{C})$. Thus, $C_{\Phi}$, the Choi matrix of $\Phi,$ is positive semidefinite. Consider the matrix $\tw{C}=\frac{C_{\Phi}+C_{\Phi}^\top}{2}$. We note that $\tw{C}$ is the Choi matrix for some map $\Psi:M_3(\mathbb{C})\rightarrow M_3(\mathbb{C})$. From this point onward we will denote $\tw{C}$ as $C_{\Psi}$. Since transposition is a positive map,  $C_{\Psi}$ is also a positive matrix and hence $\Psi$ is a completely positive map. Hence by Choi's Theorem \cite{choi2}, there exist finitely many matrices (of the appropriate size) such that, for $A\in M_3(\mathbb C)$,
\begin{equation}\label{psirep}\Psi(A)=\sum_i V_i^*A V_i.
\end{equation} 
To be clear, writing $C_{\Phi}=(C_{jk})_{j,k=1}^3$ where the $C_{ij}$ $3\times 3$ are matrices, and using $C_{jk}=C_{kj}^*$ (since $C=C^*$)

\[C_{\Psi}=\frac{C_{\Phi}+C_{\Phi}^\top}{2}=
\frac{1}{2}\begin{pmatrix}
C_{11}&C_{12}&C_{13}\\ 
C_{12}^*&C_{22}&C_{23}\\
C_{13}^*&C_{23}^*&C_{33}\\
\end{pmatrix}+
\frac{1}{2}\begin{pmatrix}
C_{11}^\top&(C_{12}^*)^\top&(C_{13}^*)^\top\\ 
C_{12}^\top&C_{22}^\top&(C_{23}^*)^\top\\
C_{13}^\top&C_{23}^\top&C_{33}^\top\\
\end{pmatrix}.
\]
In particular, 
\begin{equation} \label{eq:realmat}
C_{\Psi} = \frac{C_{\Phi} + \overline{C_{\Phi}}}{2}. 
\end{equation}
We first show that the map $\Psi$ when restricted to $\sym_3(\mathbb{R})$ is the same map as $\Phi$ restricted to $\sym_3(\mathbb{R})$. Let $E_{jk}$ be the standard matrix basis elements and note the following basis for the symmetric complex matrices,
$\{\frac{E_{jk}+ E_{kj}}{2}: 1\le j\le k \le 3\}$.
For $i,j=1,2,3,$ $\Phi(E_{jk}+E_{kj})=C_{jk}+C_{jk}^*\in \sym_3(\mathbb{R})$ by definition as seen from \eqref{Phimap}. Hence 
\[C_{jk}+C_{jk}^*=(C_{jk}+C_{jk}^*)^\top.\]
Thus,
\[\begin{split} \Psi(E_{jk}+E_{kj})=&\frac{C_{jk}+(C_{jk}^*)^\top}{2} + \frac{C_{jk}^*+C_{jk}^\top}{2}\\  = & \frac{C_{jk}+C_{jk}^*}{2}+\frac{(C_{jk}+C_{jk}^*)^\top}{2}\\
 =&C_{jk}+C_{jk}^*=\Phi(E_{jk}+E_{kj}). \end{split}\]
Hence,  
\[\Psi|_{\sym_{3}(\mathbb{R})}=\Phi|_{\sym_{3}(\mathbb{R})}.\]
 
By \eqref{eq:realmat} $C_{\Psi}$ is a real symmetric matrix. Since $C_{\Psi}$ is positive it has a factorization into two real matrices. This is equivalent to the fact that $C_{\Phi}=\sum_{i} w_i^\top w_i$ where each $w_i$ is a $1\times 9$ matrix with real entries. Write $w_i=(x_1^i,x_2^i,x_3^i)$ where each $x_j^i$ is a $1\times 3$ matrix. For $1\le i\le 3,$ form the $3\times 3$ matrices $W_i$ whose $j$-th row is $x_j^i$. Note that $\Psi(E_{j,k})=\sum_{i} W_i^\top E_{j,k}W_i$ and by linearity $\Psi(A)=\sum_{i} W_i^\top A W_i$

Hence, the matrices $V_i$ in the representation of $\Psi$ in \eqref{psirep} can be replaced by real matrices $W_i$ and 
\[\Psi(A)=\sum_i W_i^\top A W_i.\]
Since  \[\Psi|_{\sym_{3}(\mathbb{R})}=\Phi|_{\sym_{3}(\mathbb{R})},\]
this is a contradiction of Theorem \ref{notcpmap}.
\end{proof}

\begin{prop} \label{prop:counter} For each $c,d>0$ there exists  a 3-isometric 2-tuple of invertible operators $(T,S)$ in the a class $\mathfrak{F}_{c,d}$ such that the pencil $\hat{Q}_{T,S}$ does not factor. In particular, the 2-tuple $(T,S)$ does not lift to a 2-tuple $(J_1,J_2)$ in the class $\mathfrak{J}_{c,d}$.
\end{prop}
\begin{proof} Given $c,d>0$, consider the following basis for $\sym_{3}(\mathbb{C})$,
\begin{equation}\label{basischoice}
\begin{split}
&B_{0,1}=\begin{pmatrix}0&c & 0\\ c & 0& 0 \\ 0 & 0 & 0\end{pmatrix};\ \ B_{1,0}=\begin{pmatrix}0&0 & d\\ 0 & 0& 0 \\ d & 0 & 0\end{pmatrix};\ \ B_{1,1}=\begin{pmatrix}0&0 & 0\\ 0 & 0& cd \\ 0 & cd & 0\end{pmatrix};\\
&B_{0,0}=\begin{pmatrix}1&0 & 0\\ 0 & 0& 0 \\ 0 & 0 & 0\end{pmatrix};\ \ B_{0,2}=\begin{pmatrix}0&0 & 0\\ 0 & c^2& 0 \\ 0 & 0 & 0\end{pmatrix};\ \ B_{2,0}=\begin{pmatrix}0&0 & 0\\ 0 & 0& 0 \\ 0 & 0 & d^2\end{pmatrix}.
\end{split}
\end{equation}
We note $B_{0,0}=I-\frac{1}{c^2}B_{0,2}-\frac{1}{d^2}B_{2,0}$.
By Proposition \ref{notcpmap} there exists a unital positive but not completely positive linear map $\Phi:\sym_3(\mathbb C)\to M_3(\mathbb C)$. Thus,
\begin{equation}\label{Phipencil}0\preceq\Phi\left(\begin{pmatrix}1&\alpha c & \beta d\\ \alpha c & \alpha^2c^2& \alpha \beta c d \\ \beta d & \alpha \beta c d & \beta^2d^2\end{pmatrix}\right)=\Phi\left(\sum_{0\leq i+j\le 2}\alpha^i\beta^jB_{i,j}\right)=\sum_{0\leq i+j\le 2}\alpha^i\beta^j\Phi(B_{i,j})=\hat{Q}_{\Phi}(\alpha,\beta).
\end{equation}
Here we have used the notation in Lemma \ref{factorlemma}. By Lemma  \ref{factorlemma} the canonical pencil $\hat{Q}_{\Phi}(\alpha,\beta)$ does not factor since $\Phi$ is not a completely poistive map. Let 
\[Q=I+\sum_{0<i+j\leq 2}\tw{B}_{i,j}\]
where 
\[\tw{B}_{i,j}=\Phi(B_{i,j}).\]
Note
\[Q(\alpha,\beta)-\frac{1}{c^2}\tw{B}_{0,2}-\frac{1}{d^2}\tw{B}_{2,0}=\sum_{0\leq i+j\le 2}\alpha^j\beta^k\Phi(B_{i,j})=\hat{Q}_{\Phi}(\alpha,\beta).\]
By Proposition  \ref{shiftconstruct}, since $\hat{Q}(\alpha,\beta)\succeq 0$ we can construct a 2-tuple $(T,S)$ in the class $\mathfrak{F}_{c,d}$ such that $\hat{Q}_{T,S}(\alpha,\beta)$ does not factor. 
By Theorem \ref{liftthm}, the  2-tuple $(T,S)$ does not lift. 
\end{proof}

\subsection{Strengthening the Counter-Example}
While the counter-example of Propostion \ref{prop:counter} answers the natural question of whether 2-tuples $T$ in $\mathfrak{F}_{c,d}$ always lift to a 2-tuple $J$ in the class $\mathfrak{J}_{c,d}$, we will  actually construct a stronger counter-example. Given a quadratic pencil which does not factor we will construct a 2-tuple of commuting 3-isometries that does not lift to a 2-tuple $J$ in any of the classes $\mathfrak{J}_{c,d}$. Let 
\begin{equation}\label{non-factoringQ}
Q(\alpha, \beta)=\sum_{0\leq i+j\leq 2} \alpha^i\beta^j B_{ij}\succeq 0 \text{ for all }(\alpha,\beta)\in \mathbb{R}^2
\end{equation}
 be a not necessarily monic quadratic pencil with $B_{ij}\in\sB(H)$ which does not factor. The existence of such objects is given by Proposition \ref{nofactor}. We begin  with the following lemma. 
\begin{lemma}\label{stillnofactor}
If $Q(\a,\b)$ does not factor in the form
\[Q(\alpha,\beta)=(V_0+\alpha V_1+\beta V_2)^*(V_0+\alpha V_1+\beta V_2)\]
and if $\Gamma \in \sB(H)$ is positive semidefinite, then  $Q(\a,\b)-\Gamma$ does not factor in the form
\[Q(\alpha,\beta)-\Gamma=(W_0+\alpha W_1+\beta W_2)^*(W_0+\alpha W_1+\beta W_2).\]
\end{lemma}
\begin{proof}
We prove the contrapositive. Accordingly, suppose 
\[Q(\alpha,\beta)-\Gamma=(W_0+\alpha W_1+\beta W_2)^*(W_0+\alpha W_1+\beta W_2),\]
in which case
\[Q(\alpha,\beta)=(W_0+\alpha W_1+\beta W_2)^*(W_0+\alpha W_1+\beta W_2)+\Gamma.\]
Since, $\Gamma\succeq 0$, there exists $\Delta\in \sB(H)$ such that $\Gamma =\Delta^*\Delta$.
Hence, 
\[Q(\alpha,\beta)=\left ( \begin{pmatrix}W_0\\ \Delta
\end{pmatrix}+\a \begin{pmatrix}W_1\\ 0\end{pmatrix}+\b \begin{pmatrix}W_2\\ 0\end{pmatrix}\right)^*\left ( \begin{pmatrix}W_0\\ \Delta
\end{pmatrix}+\a \begin{pmatrix}W_1\\ 0\end{pmatrix}+\b \begin{pmatrix}W_2\\ 0\end{pmatrix}\right).\]
\end{proof}

We now show there exists a monic pencil $Q(\a,\b)$ such that $Q(\alpha,\beta)-\frac{1}{c^2}B_{2,0}-\frac{1}{d^2}B_{0,2}$ does not factor for all $c,d$ for which 
\[ Q-\frac{1}{c^2}B_{2,0}-\frac{1}{d^2}B_{0,2}\succeq 0\  \text{ for all }(\a,\b)\in \mathbb{R}^2.\]

\begin{thm}\label{noclassfactor} For each  $c_0,d_0>0$ there exists  a monic quadratic pencil 
\[Q(\alpha,\beta)=I+\sum_{0<i+j\leq 2}\alpha^i\beta^j B_{i_j}\]
such that
\begin{enumerate}[(i)]
\item 
\[Q(\alpha,\beta)-\frac{1}{c_0^2}B_{0,2}-\frac{1}{d_0^2} B_{2,0}\succeq 0\]
for all $(\a,\b)\in \mathbb{R}^2$
\item if $c,d>0$, then there does not exist an auxiliary Hilbert space $K$ and operators  $V_0, V_1, V_2\in \sB(H,K)$ such that 
\[Q(\alpha,\beta)-\frac{1}{c^2}B_{0,2}-\frac{1}{d^2} B_{2,0}=(V_0+\alpha V_1+\beta V_2)^*(V_0+\alpha V_1+\beta V_2).\]
\end{enumerate}
\end{thm}

\begin{proof} Let $Q(\alpha,\beta)$ be the non-monic matrix valued pencil that does not factor, i.e. 
\[Q(\alpha,\beta):= \hat{Q}_\Phi(\alpha,\beta)=\Phi\left(\sum_{0\leq i+j \leq 2}\alpha^i\beta^jB_{i,j}\right)\]
where $\Phi$ is the map from Proposition \ref{nofactor} and $\hat{Q}_\Phi(\alpha,\beta)$ is the pencil defined by Equation \eqref{Phipencil} in the proof of Proposition \ref{prop:counter}. The first step is to show that we can assume that $Q$ is monic and that there exists a $\delta>0$ such that 
\[
  Q(\alpha,\beta)\succeq \delta I
\]
for all $\alpha,\beta\in\mathbb R$. For an operator $A\in \mathscr{B}(H)$ the notation $A\succeq 0$ will mean that for all $x\in H$
\[\langle Ax,x\rangle \ge 0.\]
We start by considering the following pencil 
\[Q_{\varepsilon}(\alpha,\beta)=Q(\alpha,\beta)+\varepsilon I\succ 0.\]
Here we need to choose $\varepsilon>0$ so that the $Q(\alpha,\beta)+\varepsilon I$ still does not factor. By Lemma \ref{factorlemma} $Q(\alpha,\beta)$ will factor if and only if the map $\Phi$ is completely positive. The map $\Phi$ is completely positive if and only if its Choi matrix $C_\Phi$ is positive semidefinite by Lemma \ref{lem:choi}. Since $\Phi$ is a unital map, and by definition of $Q(\alpha,\beta)$, we will have that $Q(\alpha,\beta)+\varepsilon I$ will not factor if $C_{\Phi}+\varepsilon I$ is not positive.  Since $C_\Phi$ is not  positive in the first place, we simply need to pick an $\varepsilon>0$ small enough so that $C_\Phi+\varepsilon I$ is not positive.
We note that 
\[Q(\alpha,\beta)=\Phi\left(\begin{pmatrix}1&\alpha c_0&\beta d_0\\\alpha c_0 &\alpha^2 c_0^2& \alpha \beta c_0 d_0\\
\beta d_0& \alpha \beta c_0 d_0& \beta^2 d_0^2
\end{pmatrix}\right)\]
where $c_0$ and $d_0$ come from the choice of basis as in \eqref{basischoice}.
Since $\Phi$ is a unital map
\[Q_\varepsilon(\alpha,\beta)=\Phi\left(\begin{pmatrix}1+\varepsilon&\alpha c_0&\beta d_0\\\alpha c_0 &\alpha^2 c_0^2+\varepsilon& \alpha \beta c_0 d_0\\
\beta d_0& \alpha \beta c_0 d_0& \beta^2 d_0^2+\varepsilon
\end{pmatrix}\right).\]
Let 
\[Q_\varepsilon(\alpha,\beta)=\sum_{0\leq i+j \leq 2}\alpha^i\beta^j\tw{B}_{i,j}. 
\]
In particular
\[Q_\varepsilon(0,0)=\tilde{B}_{00}\succeq \varepsilon \succ 0.\]
Let  $\Delta=B_{0,0}^{-\frac{1}{2}}\succeq 0$ and note that 
\[\tw{Q}_\varepsilon(\alpha,\beta):=\Delta^*[Q(\alpha,\beta)+\varepsilon I]\Delta\succ 0\] and is monic. Now choose a $\delta>0$ such that $\varepsilon\Delta^*\Delta\succeq \delta I$. Hence $\tw{Q}_{\varepsilon}(\a,\b)$ is monic and $\tw{Q}_{\varepsilon}(\a,\b)\succeq \delta I$.

With our assumptions validated from this point on we will assume we have a \emph{monic} matrix pencil $Q(\alpha,\beta)$ such that
\[Q(\alpha,\beta)\succeq \delta I\]
for all $(\alpha,\beta)\in \mathbb{R}^2$. 
Let 
\[Q(\alpha,\beta)=I+\sum_{0<i+j\leq 2}\alpha^i\beta^jB_{i,j}.\]
For all $(c,d)\in \mathbb{R}^2$ such that 
\[\delta I\succeq \left(\frac{1}{c^2}{B}_{0,2}+\frac{1}{d^2} {B}_{2,0}\right)\]
the pencil $Q$ is monic,
\[Q(\alpha,\beta)-\frac{1}{c^2}{B}_{0,2}+\frac{1}{d^2} {B}_{2,0}\succeq 0, \] 
and does not factor by Lemma \ref{stillnofactor}.
\end{proof}
We summarize in the following proposition. 
\begin{prop}
There exists $c_0,d_0>0$ and a 3-isometric 2-tuple of invertible operators $(T,S)$ in the class $\mathfrak{F}_{c_0,d_0}$ such that $(T,S)$ does not lift to any 2-tuple $J$ in any class $\mathfrak{J}_{c,d}$.
\end{prop}
\begin{proof}
The proof follows from an application of Propositions \ref{shiftconstruct} and \ref{bothfactor} and Theorem \ref{noclassfactor}.
\end{proof}
\section{Spectral Considerations and 3-Symmetric Operator Tuples}
\label{sec:spec}
Given a 2-tuple of 3-isometries in a class $\mathfrak{F}_{c,d}$ that lifts to a 2-tuple of commuting Jordan operators we will first show some control over the joint spectrum of the Jordan 2-tuple.  Secondly,  we will establish, by a holomorphic functional calculus argument, a lifting theorem analogous to Theorem \ref{liftthm} holds for 3-symmetric 2-tuples.

\subsection{Spectral Considerations}
 Let $\sigma_{\tay}(T)$ denote the Taylor spectrum of the tuple $T$ of operators on a Hilbert space. For an inviting exposition of the Taylor joint spectrum see \cite{Curto}.

\begin{prop}
 \label{prop:spectral inclusion}
   Suppose $T$ is a 2-tuple of invertible operators and $c,d>0$. If $T$ lifts to a 2-tuple $J\in \mathfrak{J}_{c,d}$, then $\sigma_{\tay}(T)\subset \sigma_{\tay}(J)$.  Moreover, in this case there exists a 2-tuple $\mathscr{J} \in \mathfrak{J}_{c,d}$ such that $T$ lifts to $\mathscr{J}$ and $\sigma_{\tay}(T)=\sigma_{\tay}(\mathscr{J})$. 
\end{prop}

Let $U=(U_1,U_2)$ be the unitary commuting tuple appearing in ${J}=({J}_1,{J}_2)$. By the form of ${J}$ it is easy to see,
\[\sigma(U_i)=\sigma({J}_i).\]
However a result involving the Taylor spectrum of $U$ and $J$ can be achieved.

\begin{prop}\label{specUequalspecJ} For Jordan 2-tuple of the form \eqref{eq:Jforms}
\[\sigma_{\tay}(U)=\sigma_{\tay}(J)\] where $U=(U_1,U_2)$ is the 2-tuple of unitary operators appearing in $J=(J_1,J_2)$. 
\end{prop}

\begin{proof}
By Proposition \ref{prop:spectral inclusion},  $\sigma_{\tay}(U)\subset \sigma_{\tay}(J)$. On the other hand,  as seen in \cite{Curto}, for operators  $A$, $B$ and $C$ on Hilbert space, 
\[ \sigma_{\tay}\left ( \begin{pmatrix}A& C\\ 0 & B\end{pmatrix}\right)\subseteq \sigma_{\tay}(A)\cup \sigma_{\tay}(B).\] 
In our case this shows that $\sigma_{\tay}(J)\subseteq \sigma_{\tay}(U)$ and the proof is complete.
\end{proof}

The proof of Propostion \ref{prop:spectral inclusion} occupies the remainder of this subsection and is broken down into a series of subresults.

For a compact set $K$, let $\co(K)$ denote the convex hull of $K$. If $K\subset \mathbb C^n$ is compact, then, by Caratheodory's Theorem,  $\co(K)$ is also compact (and hence closed).  For a closed convex set $K,$ let $\ext(K)$ denote the set of extreme points of the $K$.

\begin{lemma}
\label{lem:ext of T}
  The set of extreme points of $\co(\mathbb T^2)$ is $\mathbb T^2$.
\end{lemma}

\begin{proof} 
 The convex hull of a cartesian product is the cartesian product of the convex hulls.  The set of extreme  points of a cartesian product is the cartesian product of the extreme points. Since the extreme points of $\co(\mathbb T)=\mathbb T$ the result follows. 
\end{proof}

\begin{lemma} If $K$ is a compact subset of $\mathbb{T}^2\subset \mathbb{C}^2,$ then 
\[\ext(\co(K))=K. \]
\end{lemma} 

\begin{proof} 
Since $K\subset\mathbb T^2$, if $z\in K$, then $z$ is an extreme point of $\co(\mathbb T^2)$ by Lemma \ref{lem:ext of T} and therefore of $\co(K)$.  Hence $K\subset \ext(\co(K))$.  On the other hand, $\ext(\co(K))\subset K$ for any compact subset $K$  of $\mathbb C^n$. 
\end{proof}

\begin{definition}
The joint approximate point spectrum for a 2-tuple $T$ is defined to be the set of points $\lambda\in \mathbb{C}^2$ such that there exist unit vectors $\{x_k\}$ such that 
\[\|(T_i-\lambda_i)x_k\|\rightarrow 0 \text{ for }i=1, 2.\]
We denote joint approximate point spectrum as $\sigma_{ap}(T).$
\end{definition}

The following two lemmas are well known. Among the many references, see \cite{Curto,Choi}. The theorem following these lemmas can be found in a paper of Wrobel \cite{Wrobel}.

\begin{lemma}
  \label{lem:ap in T2}
 The approximate point spectrum of a commuting tuple $T$ of operators on Hilbert space  lies in the Taylor spectrum of $T$.
\end{lemma}

\begin{lemma}
  The Taylor spectrum of a commuting tuple $T$ of operators on Hilbert space is nonempty and compact. 
\end{lemma}

\begin{thm}
\label{thm:Wrobel} If $T$ is a commuting tuple of operators on Hilbert space, then
\[\ext(\co(\sigma_{\tay}(T)))=\ext(\co(\sigma_{ap}(T))).\]
\end{thm}

\begin{lemma}
 \label{lem:ap T vs J}
 Suppose $T$ is a commuting 2-tuple of invertible operators on a Hilbert space $H$  and $c,d>0$ and  $T$ lifts to a 2-tuple $J\in \mathfrak{J}_{c,d}$ acting on the Hilbert space $K$,  i.e. there is an isometry $V:H\to K$ such that \[VT^\alpha =J^\alpha V\]
for every multi-index $\alpha$. If $\lambda \in \sigma_{ap}(T),$ then $\lambda\in \sigma_{ap}(J);$  i.e., $\sigma_{ap}(T) \subset \sigma_{ap}(J)$.
\end{lemma}

\begin{proof}For $i=1,2,$ 
\[V(T_i-\lambda_i)=(J_i-\lambda_i)V.\]
If $\|(T_i-\lambda_i)x_k\|\rightarrow 0$ as $k\rightarrow \infty$, then $\|V(T_i-\lambda_i)x_k\|\rightarrow 0$ as $k\rightarrow \infty$ since $V$ is an isometry. Hence for the unit vectors $y_k=Vx_k$,
\[\|(J_i-\lambda_i)Vx_k\|\rightarrow 0\]
as $k\rightarrow \infty$.
\end{proof}

 We are now in position to show $\sigma_{\tay}(T)\subseteq J$. Since  $T_i$ and $J_i$ are invertible for $i=1,2$,  both $\sigma_{\tay}(T)$ and $\sigma_{\tay}(J)$ are subsets of $\mathbb{T}^2$, since for instance $\sigma_{\tay}(T)\subseteq \sigma(T_1)\times \sigma(T_2)\subseteq \mathbb{T}^2$.  In particular, by Theorem \ref{thm:Wrobel} and Lemma \ref{lem:ap in T2}, 
\[
 \sigma_{\tay}(A) = \ext(\co(\sigma_{\tay}(A)))=\ext(\co(\sigma_{ap}(A))) = \sigma_{ap}(A),
\]
where $A$ is either $T$ or $J$.  An application of  Lemma \ref{lem:ap T vs J} now gives $\sigma_{\tay}(T) \subset \sigma_{\tay}(J),$ completing the proof of the first part of Proposition \ref{prop:spectral inclusion}.

We will now complete the proof of Proposition \ref{prop:spectral inclusion} by  showing  that we can alter the 2-tuple $J$ so that $\sigma_{\tay}(J)\subseteq \sigma_{\tay}(T)$. We will state this as a proposition whose proof will require several lemmas and occupy the remainder of this section.  Suppose $T=(T_1,T_2)$ is a commuting tuple of invertible operators which lift to a commuting tuple of invertible operators $J=(J_1,J_2)\in \mathfrak{J}_{c,d}$ of the form \eqref{eq:Jforms} i.e. there exists an isometry $V$ such that 
\[VT_1T_2=J_1J_2 V.\]
Let $U=(U_1, U_2)$ be the tuple of unitary operators appearing in $J$. As in \cite{McCullough} we will show that each $U_i$ can be replaced with $W_i=(I-P)U_i(I-P)$, where $P$ is the  joint spectral projection for the complement of $\sigma_{\tay}(T)$.

\begin{prop}\label{equalspec} If a commuting tuple of invertible operators $T$ lifts to a commuting tuple of operators $J\in \mathfrak{J}_{c,d}$, then there exists a tuple of commuting invertible operators $\mathscr{J}=(\mathscr{J}_1,\mathscr{J}_2)\in \mathfrak{J}_{c,d}$ such that $T$ lifts to $\mathscr J$ and $\sigma_{\tay}(T)=\sigma_{\tay}(\mathscr{J})$.  
\end{prop}

Since the inclusion $\sigma_{\tay}(T)\subset \sigma_{\tay}(\mathscr J)$ has already been established, it remains to prove that $\mathscr J$ can be chosen in such a way that the reverse inclusion holds. 

Assuming $T_1$ and $T_2$ are both invertible,  by Theorem \ref{liftthm} there is a commuting 2-tuple of unitary operators $U_1$ and $U_2$ acting on a Hilbert space $F$ and an isometry $V:H\rightarrow F\oplus F\oplus F$ such that 
\[VT_1^nT_2^m=J_1^nJ_2^mV\] 
for all $m,n\in \mathbb{N}$ where the $J_i$ have $U_i$ as entries for $i=1,2$.
If $\sigma_{\tay}(T)=\mathbb{T}^2,$ then there is not much to prove since $\sigma_{\tay}(J)\subseteq \sigma_{\tay}(U)\subseteq \mathbb{T}^2$ and the proof is complete. So from this point onward we assume otherwise. 

As shown in \cite{McCullough} given an arc $A$ in the complement of the spectrum of a $3$-isometry T ($\sigma(T)\subseteq \mathbb{T}$), there is a holomorphic function $f$ such that $|f|\geq 1$ on the arc $A$ and $|f|<1$ on and inside $\Gamma$, where $\Gamma$ is a curve containing the spectrum. 

Let $\overline{\mathbb D}$ denote the closed unit disk, $\{z\in\mathbb C: |z|\le 1\}$, in the complex plane $\mathbb C$. 

\begin{lemma}\label{f_p} Let $p=(e^{i\theta_1}, e^{i\theta_2})$ be a point of $\mathbb{T}^2$ in the complement of the Taylor spectrum of $T$. If $\Omega_i,$ for $i=1,2$, are open sets  containing $\overline{\mathbb D}$ and  $2e^{i\theta_i}\notin \Omega_i,$ then there exists an open set $O_p\subset \mathbb{T}^2$ (open in the topology of $\mathbb{T}^2$) such that $O_p\cap \sigma_{\tay}(T)=\emptyset$ and a holomorphic function $f_p:\Omega_1\times \Omega_2 \rightarrow \mathbb{C}$ such that $|f_p|\geq 1$ on $O_p$ and $|f_p|<1$ on $\sigma_{\tay}(T)$. Moreover there exist holomorphic functions $f_{p_i}:\Omega_i\rightarrow \mathbb{C}$ such that 
\[f_{p}(z_1,z_2)=f_{p_1}(z_1)\cdot f_{p_2}(z_2).\]
\end{lemma}
\begin{proof} Given $p=(e^{i\theta_1}, e^{i\theta_2})\in \mathbb{T}^2$ consider the functions 
\[h_i:\Omega_i\rightarrow \mathbb{C}, \ \ h_i(z)=\frac{1}{(2-e^{-i\theta_i}z)} \text{ for }i=1,2\]
 and define $h:\Omega_1\times \Omega_2\rightarrow \mathbb{C}$ by
\[h(z_1,z_2)=h_1(z_1)\cdot h_2(z_2)=\frac{1}{(2-e^{-i\theta_1}z_1)(2-e^{-i\theta_2}z_2)}.\]
We note that $h(p)=1$ and $|h(z)|<1$ whenever $z\neq p$ and $z$ in the bidisk. Let $K$ be a compact subset of $\mathbb{T}^2$ not containing $p$ and note $|h^n|\rightarrow 0$ uniformly on $K$ as $n\rightarrow \infty.$ Hence, $|h^N(z)|<\frac{1}{2}$  for some $N$ large enough and all $z\in \sigma_{\tay}(T)$. Let $C$ be a positive number such that $1<C<2$ and let $O_p$ be an open set disjoint from the Taylor spectrum containing $p$ such that $C|h^N|\geq 1$ on $O_p$. Such an open set exists by continuity. Now define $f_p(z)=Ch^N(z)$ and note $f_p$ and $O_p$ satisfy the conditions of the lemma. It is clear there exists a $f_{p_i}$ for $i=1,2$ such that $f_p(z_1,z_2)=f_{p_1}(z_1)\cdot f_{p_2}(z_2)$. \end{proof}

We now choose $\Omega_1=\Omega_2=\frac{3}{2}\mathbb{D}$. Since each $U_i$ is unitary we can define $f_i(U_i)$ through the  holomorphic functional calculus or by the power series functional calculus. Of course both will give the same operator value for $f_i(U_i)$. At the same time we may define each $f_{p_i}(J_i)$ via the power series calculus.  It is straight forward to verify
\[f_{p_1}(J_1)=\begin{pmatrix}f_{p_1}(U_1)& cf_{p_1}^\prime(U_1)& 0 \\ 0 & f_{p_1}(U_1)& 0 \\ 0 & 0 & f_{p_1}(U_1)\end{pmatrix},\ \ \ \ \ 
f_{p_2}(J_2)=\begin{pmatrix}f_{p_2}(U_2)& 0& df_{p_2}^\prime(U_2) \\ 0 & f_{p_2}(U_2)& 0 \\ 0 & 0 & f_{p_2}(U_2)\end{pmatrix}.\]
Define $f_p(J)$ by
\[f_p(J)=f_{p_1}(J_1)\cdot f_{p_2}(J_2).\]
Similarly we may define $f_{p_i}(T_i)$ and hence $f(T)$ by the power series functional calculus as well. We note that any other functional calculus used to define $f(J)$ and $f(T)$ must agree with the values given by the power series calculus. 

Now write with respect to the decomposition $F\oplus F\oplus F$ 
\[V=\begin{pmatrix}V_2\\V_1\\V_0\end{pmatrix}.\]
\begin{lemma}\label{opencenteredatp}
 Let $p \in\mathbb T^2$ be in the complement of $\sigma_{\tay}(T)$ with $f_p$ and $O_p\subset \mathbb{T}^2$ as described in Lemma \ref{f_p}, then $E(O_p)V_\ell=0$ for $\ell=0,1,2$.
\end{lemma}
\begin{proof} We will surpress the $p$ in the notation for the functions $f_p$, $f_{p_1},$ and $f_{p_2}$, writing $f,f_1,f_2$ instead. By the holomorphic functional calculus we know $f_i^n(T_i)$ converges to zero in the operator norm since each $f_i^n$ converges to 0 uniformly on the Taylor spectrum for $T$. Since 
\[Vf_i^n(T_i)=f_i^n(J_i)V\ \ \ \ \text{ for }i=1,2,\]
$f_i^n(J_i)V$ also tends to 0 in operator norm. Hence $f^n(J)V$ also tends to 0 in the operator norm.  Let $E$ be the unique joint spectral measure for the 2-tuple $U$ such that 
\[E(A\times B)=E_1(A)E_2(B)\]
where $E_i$ is the spectral measure for $U_i$, $i=1,2$. Let $P$ be the spectral projection for $U$ corresponding to $O_p$, 
\[P=\int_{O_p} dE=E(O_p).\]

 Consider, with respect to the decomposition $K=F\oplus F\oplus F$ 
\[0\oplus 0 \oplus P=\begin{pmatrix}0 & 0 & 0 \\ 0 & 0 & 0 \\ 0 & 0 & P\end{pmatrix},\]
\[0\oplus P \oplus 0=\begin{pmatrix}0& 0 & 0 \\ 0 & P & 0 \\ 0 & 0 & 0\end{pmatrix},\]
and
\[P\oplus 0 \oplus 0=\begin{pmatrix}P& 0 & 0 \\ 0 & 0 & 0 \\ 0 & 0 & 0\end{pmatrix}.\]

Since $f^n(J)V$ converges to zero so do 
\[V^*f^{n}(J)^*(0\oplus 0 \oplus P)(0\oplus 0 \oplus P)f^n(J)V,\]
\[V^*f^{n}(J)^*(0\oplus P\oplus 0)(0\oplus P\oplus 0)f^n(J)V,\]
and 
\[V^*f^{n}(J)^*(P\oplus 0\oplus 0)(P\oplus 0\oplus 0)f^n(J)V.\]
By calculation 
\begin{align*}&f^{n}(J)^*(0\oplus 0 \oplus P)(0\oplus 0 \oplus P)f^n(J)\\&=\begin{pmatrix}f^n(U)^* & 0 & 0 \\ * & f^n(U)^*  & 0 \\ * & 0 & f^n(U)^* \end{pmatrix}\begin{pmatrix}0 & 0 & 0 \\ 0 & 0 & 0 \\ 0 & 0 & P\end{pmatrix}\begin{pmatrix}f^n(U) & * & * \\ 0 & f^n(U) & 0 \\ 0 & 0 & f^n(U)\end{pmatrix}\\
&=\begin{pmatrix}0 & 0 & 0 \\ 0 & 0 & 0 \\ 0 & 0 & f^n(U)^*Pf^n(U)\end{pmatrix}. 
\end{align*}
It follows that $Pf^n(U)V_0$ tends to $0$ in operator norm. However, $Pf^n(U)f^n(U)P=f^{*n}(U)Pf^n(U),$  since $P$ is the spectral projection associated with $U$. Consequently,  
\[V_0^*P|f^n|^2PV_0=V_0^*f^n(U)Pf^n(U)V_0\stackrel{\|\cdot\|}{\longrightarrow} 0.\]
But $P|f^n|^2P\geq P$ since $|f^n|\geq 1$ on the support $O_p$ of $P$. Thus  $PV_0=0$.  Similarly,
\[V^*f^{n}(J)^*(0\oplus P\oplus 0)(0\oplus P\oplus 0)f^n(J)V\stackrel{\|\cdot\|}{\longrightarrow} 0\]
and
\[V_1^*P|f^n|^2PV_1=V_1^*f^n(U)Pf^n(U)V_1\stackrel{\|\cdot\|}{\longrightarrow} 0.\]
Hence by similar argument $PV_1=0$.  Lastly since 
\[V^*f^{n}(J)^*(P\oplus 0\oplus 0)(P\oplus 0\oplus 0)f^n(J)V\stackrel{\|\cdot\|}{\longrightarrow} 0,\] 
by using the fact that $PV_1=PV_0=0$ and arguing similarly to the previous cases we have that $PV_2=0$.
\end{proof}

\begin{lemma}\label{closedset} If $A$ is a compact subset of $\mathbb T$ such that $A\cap \sigma_{\tay}(T)=\emptyset,$ then $E(A)V_\ell=0$ for $\ell=0,1,2$.
\end{lemma}
\begin{proof}Since $A$ is covered by finitely many $O_{p_i}$, indexed by a finite set $F$ we have 
\[E(A)V_\ell\preceq E\left(\bigcup_{p_i\in F} O_{p_i}\right)V_\ell\preceq \sum_{p_i\in F} E(O_{p_i})V_\ell\]
hence $E(A)V_\ell=0$ for $\ell=1,2$.
\end{proof}

Since the proof of the following lemma carries over from  \cite{McCullough} with only superficial modifications,  we simply state the result here. 
\begin{lemma}\label{Borelsinc}
Suppose $A_1\subseteq A_2\subseteq \ldots $ is an increasing sequence of Borel subsets of $\mathbb{T}^2$ and let $A=\cup_jA_j$. If $E(A_j)V_\ell=0$ for all $j$ and $\ell=0,1,2$, then $E(A)V_\ell=0$. 
\end{lemma}

The complement of $\sigma_{\tay}(T)$ can be written as an increasing sequence of closed (compact) sets. By an application of Lemmas \ref{Borelsinc} and \ref{closedset}
\[E(\sigma_{\tay}(T)^c)V_\ell=0, \ \ \ \ \text{for }\ell=0,1,2.\] 

Let $P=E(\sigma_{\tay}(T)^c).$ Each $W_i=(I-P)U_i(I-P)$ is  unitary and
\[\mathscr{J}_1=\begin{pmatrix}W_1& cW_1& 0 \\ 0 & W_1& 0 \\ 0 & 0 & W_1\end{pmatrix},\ \ \ \ \ 
\mathscr{J}_2=\begin{pmatrix}W_2& 0&dW_2 \\ 0 & W_2& 0 \\ 0 & 0 & W_2\end{pmatrix}\]
have the appropriate form.  Finally,  by Proposition \ref{specUequalspecJ}, $\sigma_{\tay}(\mathscr J)=\sigma_{\tay}(W)\subseteq \sigma_{\tay}(T)$. 

\subsection{3-Symmetric Operators Tuples} \label{sec:3sym} 
We will now go more in depth into using the holomorphic functional calculus for $T$ and $J$. For $i=1,2$ let $\Omega_i$ be a simply connected open subset of the plane. While the power series functional calculus was sufficient previously, in the forth coming section we will need to consider logarithms  and a power-series approach is not viable. Given a 2-tuple of commuting operators $T=(T_1,T_2)$ with each $\sigma(T_i)\subseteq \Omega_i$, let $g_i$, for $i=1,2$,  be analytic functions. By use of the holomorphic functional calculus we can define the operators $g_i(T_i)$. By Runge's Theorem there is a sequence of polynomials $(s_{i,n})$ which converge uniformly on compact subsets of $\Omega_i$ to $g_i$ for both $i=1,2$.  The sequences of operators  $s_{i,n}(T_i)$ converge in norm to $g_i(T_i)$ for $i=1,2,$ by the standard properties of the holomorphic functional calculus. Consider a 2-tuple of operators $J=(J_1,J_2)$ of the forms  \eqref{eq:Jforms} with $\sigma(U_i)\subset \Omega_i$ for $i=1,2$, where each $\Omega_i$ is an open simply connected subset of $\mathbb{C}$. 
For the analytic functions $g_i$ defined on $\Omega_i$ for $i=1,2$, with polynomials $(s_{i,n})$ converging uniformly,
\[g_1(J_1)=\lim s_{1,n}(J_1)=\begin{pmatrix}g_1(U_1) & cU_1g_1^\prime( U_1)& 0\\ 0 & g_1(U_1) & 0\\ 0 & 0 & g_1(U_1)\end{pmatrix},\]
\[g_2(J_2)=\lim s_{2,n}(J_2)=\begin{pmatrix}g_2(U_2) & 0&dU_2g_2^\prime(U_2)\\0 & g_2(U_2) & 0\\ 0 & 0 & g_2(U_2)\end{pmatrix}.\]
For a normal operator $T$ the operator $g_i(T)$ is normal as well. Moreover, the spectrum of $g_i(T)$ is given by the spectral mapping theorem as $g_i(\sigma(T))$. Hence, given a tuple $J=(J_1, J_2)$ and holomorphic functions $g_1$ and $g_2$ we have a formula for $g_1(J_1)$ and $g_2(J_2)$ as well as their respective spectra. 

To get some information about the individual spectra, we will use the projection property for the Taylor joint spectrum.  As seen in Curto \cite{Curto}, let  $A$ and $B$ be a $n$-tuple and $k$-tuple respectively i.e. $A=(A_1,\ldots,A_n)$ and $B=(B_1,\ldots,B_k)$. Let $(A,B)$ denote the tuple $(C_1,\ldots, C_{n+k})$ where \[C_{i}=A_i \text { for }i=1,\ldots n\] and \[C_{i}=B_{i-n} \text{ for }i=n+1,\ldots, n+k.\] The projection property for the Taylor joint spectrum is as follows, 
\[\pi_{1,\ldots,n}\sigma_{\tay}(A,B)=\sigma_{\tay}(A)\] 
and 
\[\pi_{n+1,\ldots, n+k}\sigma_{\tay}(A,B)=\sigma_{\tay}(B)\]
where we define $\pi_{1,\ldots,n}:\mathbb{C}^n\times \mathbb{C}^k\rightarrow \mathbb{C}^n$, $(z_1,\ldots, z_n, z_{1+n},\ldots, z_{n+k})\mapsto (z_1,\ldots z_n)$ and similarly for $\pi_{n+1,\ldots, n+k}$. For us this projection property  implies 
\[\pi_{i}\sigma_{\tay}(T_1,T_2)=\sigma_{\tay}(T_i)=\sigma(T_i)\]
for $i=1,2$. In the context of Proposition \ref{equalspec}, if $T=(T_1,T_2)$ lifts to a tuple $J\in\mathfrak{J}_{c,d},$ then there exists a Jordan tuple $\mathscr{J}\in\mathfrak{J}_{c,d}$ such that 
\[\sigma_{\tay}(\mathscr{J})=\sigma_{\tay}(T).\]
Since $\sigma_{\tay}(\mathscr{J}_1,\mathscr{J}_2)=\sigma_{\tay}(T_1,T_2),$ by the projection property, 
\[\sigma(\mathscr{J}_i)=\pi_{i}\sigma_{\tay}(\mathscr{J}_1,\mathscr{J}_2)=\pi_{i}\sigma_{\tay}(T_1,T_2)=\sigma(T_i),\]
for $j=1,2.$ Let $U=(U_1,U_2)$ be the unitary commuting tuple appearing in $\mathscr{J}=(\mathscr{J}_1,\mathscr{J}_2)$. Since it will be of relevance in the exposition to follow we recall for the reader the equality
\[\sigma(U_i)=\sigma(\mathscr{J}_i).\]

\begin{definition}A tuple of operators $\T=(\T_1,\T_2)$ will be called a commuting 3-symmetric tuple if there exist bounded operators $B_{j,k}$ such that, 
\[\exp(is_2\T_2)^*\exp(is_1\T_1)^*\exp(is_1\T_1)\exp(is_2\T_2)=I+\sum_{0<j+k\leq 2} s_1^js_2^k B_{j,k}\]
for all $(s_1,s_2)\in \mathbb{R}^2$.
\end{definition}
\noindent It is clear that if $\T=(\T_1,\T_2)$ is a commuting 3-symmetric tuple, then $T=(e^{i\T_1},e^{i\T_2})$ is a 3-isometric tuple. 

Consider commuting 2-tuples of 3-symmetric operators $(\T_1,\T_2)$ whose spectra lie in $[a_1,b_1]$ and $[a_2,b_2]$ respectively. We note that the Taylor joint spectrum for $(\T_1,\T_2)$ must be contained in $[a_1,b_1]\times[a_2,b_2]$. Let $G(z)=\exp(iz)$ and let $S_i=G([a_i,b_i])$. Suppose the length of each $[a_i,b_i]$ is strictly less than $2\pi.$ In this case $S_i$ is a proper subset of the unit circle $\mathbb{T}$. For each $i$ there exists $\Omega_i \supset [a_i,b_i]$ and $\Omega_{*i}\supset S_i$, open simply connected subsets of $\mathbb{C}$ such that  
\[G_1=G|_{\Omega_1}:\Omega_1\rightarrow \Omega_{*1}\]  
\[G_2=G|_{\Omega_2}:\Omega_2\rightarrow \Omega_{*2}\]
are bi-analytic.
For the operator 2-tuple of commuting 3-symmetric operators $\T=(\T_1,\T_2)$ with $\sigma(\T_i)\subseteq [a_i,b_i]$ the operators $G_i(T_i)$ are defined by the holomorphic functional calculus and $\sigma(G_i(T_i))\subseteq S_i\subset \mathbb{T}$. Let $T_i=G_i(\T_i)$ and suppose the commuting 3-isometric 2-tuple $T=(T_1,T_2)$ lifts, i.e. there exists an isometry $V$ and a Jordan tuple $J$ such that 
\[VT_1^nT_2^m=J_1^nJ_2^mV.\]
 By Proposition \ref{equalspec} and the projection property there exist unitary operators $W_1$ and $W_2$ and an isometry $V$ such that
\[ VT_1=\begin{pmatrix}W_1& cW_1& 0 \\ 0 & W_1& 0 \\ 0 & 0 & W_1\end{pmatrix}V=J_1V\]
\[VT_2=\begin{pmatrix}W_2& 0&dW_2 \\ 0 & W_2& 0 \\ 0 & 0 & W_2\end{pmatrix}V=J_2V\]
where $\sigma(W_i)=\sigma(T_i)$. 
Again each $G_i$ is bi-analytic in the neighborhood of the spectrum of each $J_i$ hence 
\begin{equation}\label{logs}
\begin{split}
&V\T_1=VG_1^{-1}(T_1)=G^{-1}_1(J_1)V\\
&V\T_2=VG_2^{-1}(T_2)=G^{-1}_2(J_2)V.
\end{split}
\end{equation}
Let $A_i=G_i^{-1}(W_i)$ and note $(G_i^{-1})^\prime(W_i)=-iW_i^*$.
Hence, 
\[V\T_1=\begin{pmatrix}A_1& -ic& 0 \\ 0 &A_1& 0 \\ 0 & 0 & A_1\end{pmatrix}V\]
\[V\T_2=\begin{pmatrix}A_2& 0&-id \\ 0 & A_2& 0 \\ 0 & 0 & A_2\end{pmatrix}V.\]

If the spectrum of each $\T_i$ does not have length less than $2\pi$ we can do the same analysis on the operators $\tw{\T}_i=t_i\T_i$ where each $t_i$ is chosen so that $\sigma(\tw{\T}_i)$ is of length less than $2\pi$. As shown in \cite{McCullough} these are also 3-symmetric operators. The Taylor spectrum of the 3-symmetric tuple $\tw{\T}=(\tw{\T}_1,\tw{\T}_2)$ is contained in some $[a_1,b_2]\times [a_2,b_2]$ where each $[a_i,b_i]$ is of length less than $2\pi$. Again $\tw{T}=(\exp(\i\tw{\T}_1),\exp(\i\tw{\T}_2))$ is a 3-isometric tuple and suppose they lift by Theorem \ref{liftthm}, i.e. there exists an isometry $V$ and Jordan tuple $\tw{J}$ such that
\[V\tw{T}_2^n\tw{T}_1^m=\tw{J}_1^m\tw{J}_2^nV\]
and moreover 
\[V\tw{T}_i=\tw{J}_iV.\]
By applying the same argument as in \eqref{logs} we have
\[V\tw{\T}_i=\tw{\J}_iV\]
and thus
\[V\T_i=\frac{1}{t_i}\T_iV.\]

By noting that $\T$ and $T=\exp(i\T)$ share the same operator pencil, we see that the 3-symmetric version of Theorem \ref{liftthm}, stated below for the reader's convenience, holds.
\begin{thm}\label{liftthm-again} Tuples of 3-symmetric operators $(\T_1, \T_2)$ will lift to a 2-tuple $(\J_1, \J_2)$ of the forms
\[\J_1=\begin{pmatrix}A_1& -ic& 0 \\ 0 &A_1& 0 \\ 0 & 0 & A_1\end{pmatrix}\ \ \ \ \ \J_2=\begin{pmatrix}A_2& 0&-id \\ 0 & A_2& 0 \\ 0 & 0 & A_2\end{pmatrix}\]
if and only if the polynomial
 \footnotesize
\[\hat{Q}_{\T}(\alpha,\beta)=I+\a{B}_{1,0}+\b{B}_{0,1}+\a\b {B}_{1,1}+\a^2{B}_{2,0}+\b^2{B}_{0,2}-\frac{1}{c^2}{B}_{2,0}-\frac{1}{d^2}{B}_{0,2} \succeq 0\]
\normalsize
factors in the form, 
\[\hat{Q}_{\T}(\alpha,\beta)=(V_0+\alpha V_1 + \beta V_2)^*(V_0+\alpha V_1 + \beta V_2)\]
for some operators $V_0$, $V_1$ and $V_2$ in $\sB(H)$. 
\end{thm}
\begin{proof} By the arguments in this section, we need only prove one statement, that with $T_1=\exp(i\T_1)$ and $T_2=\exp(i\T_2)$ that $T=(T_1,T_2)\in~\mathfrak{F}_{c,d}$ for some $c,d>0$. However, this is rather simple. For $(s_1,s_2)\in \mathbb{R}^2$, let  
\[Q(s_1,s_2):=I+\sum_{0<j+k\leq 2} s_1^js_2^k B_{j,k}=\exp(is_2\T_2)^*\exp(is_1\T_1)^*\exp(is_1\T_1)\exp(is_2\T_2).\]
By definition,  
\[\exp(it_2\T_2)^*\exp(it_1\T_1)^*Q(s_1,s_2)\exp(it_1\T_1)\exp(it_2\T_2)=Q(s_1+t_1, s_2+t_2).\]
Hence by term comparison
\[\exp(it_2\T_2)^*\exp(it_1\T_1)^*B_{0,2}\exp(it_1\T_1)\exp(it_2\T_2)=B_{0,2}\] 
and 
\[\exp(it_2\T_2)^*\exp(it_1\T_1)^*B_{2,0}\exp(it_1\T_1)\exp(it_2\T_2)=B_{2,0}.\] 
If $c$ and $d$ are large enough such that 
\[I-\frac{1}{c^2}B_{0,2}-\frac{1}{d}B_{2,0}\succeq 0,\]
then
\[\exp(it_2\T_2)^*\exp(it_1\T_1)^*(I-\frac{1}{c^2}B_{0,2}-\frac{1}{d}B_{2,0})\exp(it_1\T_1)\exp(it_2\T_2)\succeq 0.\]
The existence of such $c$ and $d$ is easy enough to show, and thus
$T=(e^{i\T_1},e^{i\T_2})=(T_1,T_2)\in \mathfrak{F}_{c,d}$.
\end{proof}
In the context of Helton and Ball's conjecture \ref{3symconjecture} we have established a necessary and sufficient condition in the case $\{T_n\}$ has cardinality two. Hence, any attempt to solve this conjecture will be met with our factoring condition. 
\section{Acknowledgements}
The author would like to thank Scott McCullough whose guidance vastly improved the content and clarity of this article. 
\bibliographystyle{alpha}
\bibliography{factoringandlifts04aug2015.bbl}
\end{document}